\documentclass[a4paper,
reqno,
12pt,
intlimits,
tbtags, %
titlepage
]{amsart} %

\usepackage[T1]{fontenc}
\usepackage[utf8]{inputenc}

\usepackage{amssymb}
\usepackage{latexsym}
\usepackage{amsmath}
\usepackage{amsthm}
\usepackage{paralist}
\usepackage{microtype}
\usepackage{lmodern}

\usepackage{xspace}

\usepackage{times}

\usepackage{widows-and-orphans}

\usepackage[pagewise]{lineno}

\usepackage{hyperref}
\hypersetup{urlcolor = red, citecolor = blue}
\hypersetup{
  pdfauthor = Xuan Mao, 
  pdftitle = {Finite-time blowup in indirect chemotaxis}, 
  pdfsubject = Manuscript, 
  pdfkeywords = {Chemotaxis; indirect signal production; finite-time blowup}
}
\usepackage[top = 1in,bottom = 1in,left = 1in,right = 1in]{geometry}
\usepackage[sort]{cite}

\allowdisplaybreaks[4]
\theoremstyle{plain}
\newtheorem{theorem}{Theorem}[section]

\newtheorem{lemma}[theorem]{Lemma}
\newtheorem{proposition}[theorem]{Proposition}

\theoremstyle{definition}
\newtheorem{definition}{Definition}[section]

\theoremstyle{remark}
\newtheorem{remark}[theorem]{Remark}

\numberwithin{equation}{section}

\newcommand{\dd}{\;\mathrm{d}\xspace}

\DeclareMathOperator{\loc}{loc}

\providecommand{\um}{\ensuremath{\underline{M}}}

\providecommand{\uphi}{\ensuremath{\underline{\phi}}}
\providecommand{\upsi}{\ensuremath{\underline{\psi}}}

\providecommand{\uu}{\ensuremath{\underline{U}}}
\providecommand{\ou}{\ensuremath{\overline{U}}}

\providecommand{\uw}{\ensuremath{\underline{W}}}
\providecommand{\ow}{\ensuremath{\overline{W}}}

\providecommand{\mP}{\ensuremath{\mathcal{P}^{(\mu^\star)}}}

\providecommand{\e}{\ensuremath{\mathrm{e}}}

\begin{document}

\title[Critical mass for finite-time chemotactic collapse]{Critical mass for finite-time chemotactic collapse in the critical dimension via comparison}

\author[X.~Mao]{Xuan Mao}
\address[X.~Mao]{School of Mathematics\\ 
Hohai University\\ 
Nanjing 211100\\
Jiangsu, China}
\email[X.~Mao]{20250611@hhu.edu.cn}
\thanks{The first author has been supported by ``the Fundamental Research Funds for the Central Universities'' (No.~B250201215).}

\author[M.~Liu]{Meng Liu}
\address[M.~Liu]{Department of Applied Mathematics, Anhui University of Technology, Ma'anshan
243002, P. R. China}
\email[M.~Liu]{LMeng\_math@yeah.net}

\author[Y.~Li]{Yuxiang Li}
\address[Y.~Li]{School of Mathematics, Southeast University, Nanjing 211189, P. R. China}
\email[Y.~Li]{lieyx@seu.edu.cn}

\thanks{
The last author has been supported in part by National Natural Science Foundation of China (No. 12271092, No. 11671079). 
}



\begin{abstract}
  We study the Neumann initial-boundary value problem for the parabolic-elliptic chemotaxis system, proposed by J\"ager and Luckhaus (1992). 
  We confirm that their comparison methods can be simplified and refined, applicable to seek the critical mass $8\pi$ concerning finite-time blowup in the unit disk. 
  As an application, we deal with a parabolic-elliptic-parabolic chemotaxis model involving indirect signal production in the unit ball of $\mathbb R^4$, proposed by Tao and Winkler (2025). 
  Within the framework of radially symmetric solutions, 
  we prove that 
  \begin{itemize}
    \item if initial mass is less than $64\pi^2$, then solution is globally bounded; 
    \item for any $m$ exceeding $64\pi^2$, there exist nonnegative initial data with prescribed mass $m$ 
    such that the corresponding classical solutions exhibit a formation of Dirac-delta type singularity in finite time, termed a chemotactic collapse.
  \end{itemize}
\end{abstract}

\maketitle
\markleft{X.~Mao, M.~Liu and Y.~Li}

\section{Introduction}\label{introduce section}

This paper is concerned with finite-time blowup of radially symmetric classical solutions to the following chemotaxis models \eqref{sys: ks isp pep} accounting for indirect signal production 
\begin{align}
  \label{sys: ks isp pep}
  \begin{cases}
      u_t = \Delta u - \nabla \cdot(u\nabla v),&  x\in\Omega, t>0,\\
      0 =  \Delta v - \mu(t) + w,&   x\in\Omega,	t>0,\\
      w_t  = \Delta w - \delta w + u, &   x\in\Omega, t > 0,\\
      \partial_\nu u = \partial_\nu v = \partial_\nu w = 0 , &  x\in\partial\Omega, t >0,\\
      (u(\cdot, 0), w(\cdot,0)) = (u_0, w_0), & x\in\Omega,
  \end{cases}
\end{align}
in a ball $\Omega := B_R \subset\mathbb R^n$ centering at $0$ for some $n\in\mathbb N$ and $R>0$, 
where
  \begin{equation}
  \label{sym: mu}
    \mu(t) := \frac{1}{|\Omega|}\int_\Omega w \dd x 
    \quad\text{for } t > 0,
  \end{equation}
$\partial_\nu$ denotes the derivative with respect to the outward normal of $\partial\Omega$, 
$|\Omega|$ is the Lebesgue measure of $\Omega\in\mathbb{R}^n$,
and initial functions $u_0\in W^{1,\infty}(\Omega)$ and 
$w_0\in W^{1,\infty}(\Omega) $ are nonnegative and radially symmetric. 
Here and thereafter, $\delta\in\mathbb{R}$ is a positive constant.

The system~$\eqref{sys: ks isp pep}$, proposed and studied by Tao and Winkler~\cite{Tao2025} recently, is the J\"ager-Luckhaus variant~\cite{Jaeger1992} of its fully parabolic version
\begin{equation}
    \label{sys: ks isp ppp}
    \begin{cases}
        u_t = \Delta u - \nabla \cdot(u\nabla v), &  x\in\Omega, t>0,\\
        v_t =  \Delta v - v + w, &   x\in\Omega,	t>0,\\
        w_t  = \Delta w - w + u, &   x\in\Omega, t > 0,
    \end{cases}
\end{equation}
describing complex and realistic biological phenomena, beyond the scope of what the classical Keller-Segel model~\cite{Keller1970} can do. 
For example, \eqref{sys: ks isp ppp} models effects of phenotypical heterogeneity on population aggregation~\eqref{sys: ks isp pep}~\cite{Macfarlane2022}.
A feature of \eqref{sys: ks isp ppp} is the assumption that the population is divided into two phenotypes: 
one denoted by $u$, performing chemotactic movement towards location of higher concentration $v$ of chemical signals, 
and the other denoted by $w$, mediating signals, namely indirect signal production mechanisms.  
The model \eqref{sys: ks isp ppp} has applications to spatial memory movement of animals~\cite{Shi2021} as well. 
See the review~\cite{Winkler2025} for more biological backgrounds.

The studies of \eqref{sys: ks isp pep} can be traced back to the pioneering work~\cite{Jaeger1992} due to J\"ager and Luckhaus, 
who proposed the parabolic-elliptic chemotaxis system in the unit disk $\Omega = B_1\subset\mathbb{R}^2$
\begin{equation}
  \label{sys: JL system}
  \begin{cases}
    u_t = \Delta u - \nabla\cdot(u\nabla v), & x\in\Omega, t > 0,\\
    0 = \Delta v - \tilde\mu + u, & x\in\Omega, t > 0,\\
    \partial_\nu u = \partial_\nu v = 0, & x\in\partial\Omega, t>0,\\
    u(\cdot,0) = u_0(\cdot), & x\in\Omega,
  \end{cases}
  \quad \text{with } 
  \tilde\mu := \frac{1}{|\Omega|}\int_\Omega u_0 \dd x.
\end{equation}
They employed a comparison argument to 
detect the existence of finite-time blowup at large mass levels rigorously, i.e., $\int_\Omega u_0\dd x\gg1$, 
which will be discussed thoroughly in Section~\ref{sec: motivations}.
Their work is far-reaching~\cite{Cieslak2008,Tao2017,Winkler2019,Tello2022,Tao2025} but incomplete in the sense of capturing the critical mass $8\pi$ regarding finite-time radial blowup as Nagai done.  

Nagai~\cite{Nagai1995} introduced $2$nd-moment methods 
and found
\begin{itemize}
  \item if $\int_\Omega u_0\dd x < 8\pi$, then solutions are globally bounded;
  \item for any $m>8\pi$, there exists $\epsilon > 0$ such that for any radially symmetric and nonnegative function $u_0\in C^0(\overline\Omega)$ complying with $\int_\Omega u_0\dd x = m$ and $\int_\Omega |x|^2u_0\dd x < \epsilon$, 
  the corresponding classical solution to \eqref{sys: JL system} blows up in finite time $T_{\max} \in(0, \infty)$.
\end{itemize}
Reversely, if a radially symmetric solution of \eqref{sys: JL system} blows up, 
then it happens in finite time and the initial mass exceeds $8\pi$~\cite{Ohtsuka2007,Mao2024c}. 
Radially symmetric exploding solution exhibits a chemotactic collapse~\cite{Senba2001,Senba2004}, i.e., 
\begin{equation}
  \label{eq: u to 8pi delta}
  u(\cdot,t)\rightharpoonup m_0\delta_0(\cdot) + f(\cdot) 
  \quad \text{in }
  \mathcal{M}(\overline{\Omega}) := (C_0(\overline{\Omega}))^\prime 
  \text{ as } t\uparrow T_{\max},
\end{equation}
for some $m_0\geq8\pi$ and nonnegative, radially symmetric and integrable function $f\in C^0_{\loc}(\overline{\Omega}\setminus\{0\})$.  
Here and below, $\delta_0\in \mathcal{M}(\overline{\Omega})$ denotes Dirac-delta distribution function with unit weight at $0$.
A lot of analysis activities~\cite{Herrero1996,Senba2004,Suzuki2013,Collot2022,Mizoguchi2022} \textit{suggest} that $m_0 = 8\pi$.
We shall mention that the unique solution (up to translation and scaling) 
\begin{equation}
  \label{sym: V_0}
  V_0(x) := \frac{8}{(1+|x|^2)^2} 
  \quad \text{for } x\in\mathbb{R}^2
\end{equation}
to the stationary problem
\begin{equation}
  \label{sys: V_0}
  -\Delta \ln V_0 = V_0, \quad x\in\mathbb R^2
\end{equation}
of the Patlak-Keller-Segel system~\cite{Chen1991},
serves as a building block of studies in various chemotaxis models. 
For example, \eqref{sym: V_0} has been applied to construct subsolutions \cite{Jaeger1992,Tao2017,Winkler2019}, supersolutions~\cite{Tao2017,Winkler2019,Mao2024c},
initial data~\cite{Horstmann2001,Laurencot2019,Fujie2021} 
and special solutions~\cite{Herrero1996,Collot2022,Mizoguchi2022}.

In the direction of seeking critical mass for chemotactic collapse via maximum principle, 
Tao and Winkler~\cite{Tao2017} initialized the study on infinite-time blowup of indirect chemotaxis  
\begin{equation}
  \label{sys: ks isp peo}
  \begin{cases}
    u_t = \Delta u - \nabla\cdot(u\nabla v), &  x\in\Omega, t>0,\\
    0 = \Delta v - \mu(t) + w, &  x\in\Omega, t>0,\\
    w_t = - \delta w + u, &  x\in\Omega, t>0,
  \end{cases}
\end{equation}
in the unit disk, 
modeling cluster attack of beetles~\cite{Strohm2013}, 
where non-diffusive nesting beetles $w$ secret the signal $v$ that attracts flying beetles $u$.
Besides global solvability of \eqref{sys: ks isp peo},  
they showed that under symmetry assumptions, 
\begin{itemize} 
\item initial mass $\int_\Omega u_0\dd x < 8\pi\delta$ implies solutions remain bounded, 
\item but for any $m>8\pi\delta$, there exist initial data satisfying $\int_\Omega u_0\dd x = m$ such that solutions grow up at least exponentially~\cite{Tao2017}, 
\end{itemize}
which also exhibit a chemotaxis collapse with complete mass~\cite{Mao2024}, i.e.,
\begin{equation*}
  u(\cdot,t) 
  \rightharpoonup 
  \Big(\int_\Omega u_0\dd x\Big) 
  \cdot \delta_0(\cdot) 
  \quad\text{in } \mathcal{M}(\overline{\Omega}), 
  \text{ as } t\to\infty.
\end{equation*}

For the fully parabolic indirect chemotaxis~\eqref{sys: ks isp ppp},
Fujie and Senba~\cite{Fujie2017} established global boundedness of classical solutions to \eqref{sys: ks isp ppp} in physical spaces $n\leq3$, and 
identified a four-dimensional critical mass phenomenon:
\begin{itemize}
  \item  
  Under symmetry assumptions, $\int_\Omega u_0 < 64\pi^2$ implies the solution exists globally and remains bounded. 
  Without radial symmetry, the same conclusion holds for \eqref{sys: ks isp ppp} under certain mixed boundary conditions~\cite{Fujie2017}. 
  \item For any $m \in(64\pi^2,\infty)\setminus 64\pi^2\mathbb N$, there exist initial data satisfying 
  $\int_\Omega u_0 = m$ such that 
  solutions of \eqref{sys: ks isp ppp} under the mixed boundary conditions blow up in finite or infinite time~\cite{Fujie2019}.
\end{itemize}
In five and higher dimensions, 
authors of~\cite{Mao2025} showed that for any $m>0$, 
there exist initial data with $\int_\Omega u_0\dd x = m$ such that solutions of the system \eqref{sys: ks isp ppp} blow up in finite time with symmetry assumptions.
Existence of finite-time blowup solutions in four dimensions has been still open.

Considering volume-filling effect~\cite{Painter2002}, Ding and Wang~\cite{Ding2019} studied a general version of \eqref{sys: ks isp ppp}, where the dynamics of $u$ is rather governed by 
\begin{equation*}
  \label{eq: quasi-linear}
  u_t = \nabla\cdot(D(u)\nabla u - S(u)\nabla v), \quad x\in\Omega, t > 0.
\end{equation*}
For brevity, here we take prototypical choices of $D$ and $S$ accordingly 
\begin{equation*}
  D(\xi) = (1+\xi)^{-\alpha} 
  \quad \text{and} \quad 
  S(\xi) = (1+\xi)^{\beta -1}\xi
  \quad \text{for } \xi\geq0
\end{equation*}
with some $\alpha,\beta\in\mathbb{R}$.
They showed that $\alpha + \beta < \min\{1+2/n,4/n\}$ implies that classical solutions of the general version exist globally and remains bounded. 
Concerning critical lines of finite-time blowup,
Tao and Winkler~\cite{Tao2025} developed a robust comparison argument for a cooperative parabolic system to deal with the J\"ager-Luckhaus variant of the general version, 
and proved that if $\alpha + \beta > 4/n$ and $\beta > 2/n$ for $n\geq3$, 
then there exist radially symmetric initial data such that solutions blow up in finite time.
Mao and Li~\cite{Mao2024a} focused on the case of $\alpha + \beta > 4/n$ but $\beta < 2/n$, 
and established that $\beta<2/n$ for $n\geq2$ is sufficient for global existence of classical solutions to the general version, 
and that $\alpha + \beta > 4/n$ for $n\geq4$ warrants that for any prescribed mass $m>0$, 
one can find radially symmetric initial data $(u_0,v_0,w_0)$ with  $\int_\Omega u_0\dd x = m$ such that solutions explode in finite or infinite time.

See \cite{Laurencot2019,Laurencot2021,Painter2023,Laurencot2024,Hosono2025,Chen2025} and reference therein for related studies on indirect chemotaxis.

\subsection*{Main results}

Keeping the common feature~\eqref{eq: u to 8pi delta} of chemotactic collapse in mind,
we are ambitious to refine comparison arguments of J\"ager and Luckhaus to reach the critical mass regarding blowup.

\begin{proposition}
  \label{prop: finite-time blowup in JL system}
  Let $\Omega$ be the unit disk.
  For any $m>8\pi$, 
  there exist nonnegative, increasing and concave functions $\tilde M$ from $C^1([0,1])$ and a time $T^\star\in(0,\infty)$ such that  
  \begin{equation}
    \label{eq: boundary conditions}
    \tilde M(0) = 0
    \quad\text{and}\quad
    \tilde M(1) < m,
  \end{equation} 
  and that whenever 
  \begin{equation}
    \label{h: initial data u0}
    u_0\in C^0(\overline\Omega) 
    \text{ is radially symmetric and nonnegative}
  \end{equation}
  with 
  \begin{equation}
    \label{h: initial mass constraint of JL}
    \int_\Omega u_0 \dd x \geq m
  \end{equation}
  as well as 
  \begin{equation}
    \label{h: initial comparison on u0}
    \int_{B_{\rho^{1/2}}} u_0 \dd x \geq \tilde M(\rho) 
    \quad\text{for all } \rho\in(0,1),
  \end{equation}
  then the corresponding classical solution of \eqref{sys: JL system} blows up in finite time $T_{\max}\leq T^\star$ in the following sense:
  There exist $T_{\max } \in(0, T^{\star}]$ as well as uniquely determined functions
\begin{equation}
  \label{eq: function class of JL}
\begin{cases} 
u \in C^0(\overline{\Omega} \times[0, T_{\max })) \cap C^{2,1}(\overline{\Omega} \times(0, T_{\max })) 
\quad \text { and }\\
v \in C^{2,0}(\bar{\Omega} \times(0, T_{\max }))  \\
\end{cases}
\end{equation}
such that $u \geq 0$ in $\Omega \times(0, T_{\max })$, 
that $\int_{\Omega} v(\cdot, t) \dd x=0$ for all $t \in(0, T_{\max })$, 
that 
\begin{equation}
  \label{eq: extensibility}
  \limsup _{t \uparrow T_{\max }}\|u(\cdot, t)\|_{L^{\infty}(\Omega)}=\infty,
  \end{equation}
and that \eqref{sys: JL system} is satisfied in the classical sense in $\Omega \times(0, T_{\max })$.
\end{proposition}

\begin{remark}
  From a technical perspective, 
  besides the 2nd-moment function mentioned above 
  \begin{equation*}
    \int_\Omega |x|^2 u(x,t)\dd x 
    \quad\text{for } t\in(0,T_{\max}),
  \end{equation*} 
  introduced by Nagai~\cite{Nagai1995},
a generalized moment function 
\begin{equation*}
  \frac{1}{p} \int_0^1\bigg(\int_r^1 u(\rho,t) \rho \dd \rho\bigg)^p r \dd r 
  \quad\text{for } t\in(0,T_{\max})
\end{equation*}
was proposed by Cie\'{s}lak and Lauren\c{c}ot~\cite{Cieslak2009}
to detect finite-time blowup at large mass levels with some $p\geq1$.
Winkler~\cite{Winkler2019} established a spatially localized criterion on initial data with supercritical mass for finite-time blowup via a weighted functional
\begin{equation*}
  \int_0^{s_0}(s_0-s)^q M(s,t)\dd s 
  \quad\text{for } t\in(0,T_{\max}),
\end{equation*}   
with some $0<s_0\ll1$ and $q\gg1$, first introduced in~\cite{Winkler2018}. 
On the basis of explicit solutions~\eqref{sym: V_0},
We will construct special functions $\underline{u}\in C^0_{\loc}(\overline\Omega\times[0,T^\star))$ with the property, inter alia,
\begin{equation*}
  \underline{u}(\cdot, t) \rightharpoonup 8\pi\delta_0(\cdot) 
  \quad \text{in } \mathcal{M}(\overline{\Omega}), 
  \text{ as } t\uparrow T^\star,
\end{equation*}
to show Proposition~\ref{prop: finite-time blowup in JL system}.
\end{remark}

Our second objective is to explore finite-time blowup in the critical case $\alpha + \beta = n/4$ with particular choices $\alpha = 0$ and $\beta = 1$ (of course $n=4$), i.e., \eqref{sys: ks isp pep} in the critical dimension. 
Our main result states that initial data $(u_0,w_0)$ evolve into finite-time blowup solutions, 
provided that $u_0$ with mass exceeding $64\pi^2$ and $w_0$ with mass exceeding $16\pi^2$ aggregate near the origin in an appropriate sense.

\begin{theorem}
  \label{thm: finite-time blowup}
  Assume $\Omega = B_1 \subset\mathbb R^4$.
  Then for any $m > 64\pi^2$ and $\kappa > 0$,
  there exist nonnegative, increasing and concave functions $M^{(u)}$ and $M^{(w)}$ from $C^1([0,1])$ and a time $T\in(0,\infty)$ such that  
  \begin{equation}
    \label{eq: vanish at zero}
    M^{(u)}(0) = M^{(w)}(0) = 0,
  \end{equation}
  that 
  \begin{equation}
    \label{eq: estimates at right end}
    M^{(u)}(1) < m 
    \quad\text{and}\quad 
    M^{(w)}(1) \in \bigg(16\pi^2, 16\pi^2\cdot\frac{1+2\kappa}{1+\kappa}\bigg),
  \end{equation} 
  and that whenever 
  \begin{equation}
    \label{h: initial data}
    u_0\in W^{1,\infty}(\Omega) 
    \text{ and }
    w_0\in W^{1,\infty}(\Omega)
    \text{ are radially symmetric and nonnegative}
  \end{equation}
  with 
  \begin{equation}
    \label{h: initial mass constraint}
    \int_\Omega u_0 \dd x = m
    \quad \text{and}\quad 
    \int_\Omega w_0\dd x \leq 16\pi^2\bigg(1+\kappa+\frac{1}{\kappa}\bigg),
  \end{equation}
  as well as 
  \begin{equation}
    \label{h: initial comparison}
    \int_{B_{s^{1/4}}} u_0 \dd x \geq M^{(u)}(s) 
    \quad \text{and}\quad 
    \int_{B_{s^{1/4}}} w_0 \dd x \geq M^{(w)}(s) 
    \quad\text{for all } s\in(0,1),
  \end{equation}
  then the corresponding classical solution of \eqref{sys: ks isp pep} 
  given by Proposition~\ref{prop: local existence and uniqueness} blows up in finite time $T_{\max}\leq T$. 
\end{theorem}

\begin{remark}
  If a function triplet $(u,v,w)$ solves the system~\eqref{sys: ks isp pep} posed in the unit ball $B_1\subset\mathbb{R}^4$, 
  then for $R>0$, 
  \begin{equation*}
    (u_R, v_R, w_R) := (R^{-4}u(R^{-1}x,R^{-2}t),v(R^{-1}x,R^{-2}t),R^{-2}w(R^{-1}x,R^{-2}t))
  \end{equation*} 
  solves the system~\eqref{sys: ks isp pep} posed in $B_R\subset\mathbb{R}^4$ with $\delta$ replaced by $R^{-2}\delta$. 
  So the scaling argument above suggests that $L^2(\Omega)$ is the critical norm of $w$, analogous to $L^1(\Omega)$ of $u$. 
  And the large mass condition of $w$ in our theorem is supposed to be inessential, but of technical restriction.

Though \eqref{sys: ks isp peo} and \eqref{sys: ks isp pep} share the indirect signal production mechanism, 
singularity properties of \eqref{sys: ks isp pep} are rather akin to the two-dimensional J\"ager-Luckhaus system~\eqref{sys: JL system}, as intuitively conjectured by Fujie and Senba~\cite{Fujie2017,Fujie2019} with respect to \eqref{sys: ks isp ppp}.
\end{remark}

For completeness, 
we show boundedness under subcritical mass settings, 
as a consequence of an $\varepsilon$-regularity derived from  Lyapunov functional and Adams inequality.

\begin{proposition}
  \label{prop: varepsilon regularity}
  Assume $\Omega = B_1 \subset\mathbb R^4$.
  Let $(u,v,w)$ be a radially symmetric classical solution of \eqref{sys: ks isp pep}. 
  If there exists $\varrho\in(0,1]$ such that 
  \begin{equation}
    \label{eq: supremum on local integral of u}
    \sup_{t\in(0,T_{\max})}\int_{B_\varrho}u(\cdot,t)\dd x < 64\pi^2,
  \end{equation}
  then 
  \begin{equation}
    \label{eq: L infinity u}
    \sup_{t\in(0,T_{\max})}\|u\|_{L^\infty(\Omega)} < \infty.
  \end{equation}
  In particular, for any radial symmetric initial data $(u_0,w_0)$ with \eqref{h: initial data} and 
  \begin{equation*}
    \int_\Omega u_0\dd x < 64\pi^2,
  \end{equation*}
  the corresponding classical solution exists globally and remains bounded.
\end{proposition}

From an independent interest, 
we show that a radially symmetric solution exploding in finite time, exhibits a chemotactic collapse,
as the second consequence of the aforementioned $\varepsilon$-regularity.

\begin{proposition}
  \label{prop: chemotactic collapse}
  Assume that $\Omega = B_1\subset\mathbb{R}^4$.
  Let $(u,v,w)$ be a radially symmetric solution of \eqref{sys: ks isp pep}. 
  If $T_{\max} < \infty$, 
  then
  there exist nonnegative integrable function $f\in C^0_{\loc}(\overline{\Omega}\setminus\{0\})$ and $m_\star \geq 64\pi^2$ such that 
  \begin{equation}
    \label{eq: weak star convergence of u}
    u(\cdot, t) \rightharpoonup m_\star \delta_0(\cdot) + f(\cdot) 
    \quad \text{in } \mathcal{M}(\overline{\Omega}) 
    \text{ as } t\uparrow T_{\max}.
  \end{equation}
\end{proposition}

This paper is organized as follows. 
In Section~\ref{sec: motivations}, 
by examining the study of finite-time radial blowup in the critical dimension emanated from J\"ager and Luckhaus, 
we illustrate our ideas.
In Section~\ref{sec: finite-time of indirect chemotaxis}, 
we give the proof of Theorem~\ref{thm: finite-time blowup}.
We show Proposition~\ref{prop: varepsilon regularity} and Proposition~\ref{prop: chemotactic collapse} in Section~\ref{sec: boundedness} and Section~\ref{sec: chemotactic collapse}, respectively.

\section{A case study. Proof of Proposition~\ref{prop: finite-time blowup in JL system}}
\label{sec: motivations}

This section is devoted to introduction of our ideas to detect finite-time chemotactic collapse in the critical dimension via comparison arguments, 
as refinements of original methods due to J\"ager and Luckhaus~\cite{Jaeger1992}. 

In the pioneering work~\cite{Jaeger1992}, J\"ager and Luckhaus proposed and studied the Neumann initial-boundary value problem~\eqref{sys: JL system},  
with  
\begin{equation*}
  \int_\Omega v(\cdot, t)\dd x = 0
  \quad\text{for all } t > 0\,.
\end{equation*}
They introduced the mass distribution function 
\begin{equation}
  M(\rho, t) := \int_0^{\rho^{1/2}}ru(r,t)\dd r,
  \quad (\rho,t)\in(0,1)\times(0,T_{\max})
\end{equation}
acting on the radially symmetric classical solution $(u,v)$ with the maximal existence time $T_{\max}\in(0,\infty]$ of \eqref{sys: JL system}.
Then $M$ solves the scalar parabolic problem 
\begin{subequations}
  \label{sys: Theta}
  \begin{equation} 
  \Theta M = 0 
  \quad\text{for all } (\rho,t)\in(0,1)\times(0,T_{\max}),
  \end{equation}
  with Dirichlet boundary conditions, 
  \begin{equation}
    M(0,t) = 0
    \quad\text{and}\quad 
    M(1,t) = \frac{m}{2\pi}
    \quad\text{for all } t\in(0,T_{\max})
  \end{equation}
  and obvious initial data 
  \begin{equation}
    \label{sym: M(rho,0)}
    M(\rho,0) := \int_0^{\rho^{1/2}}ru_0(r)\dd r, 
    \quad \rho\in(0,1)\,.
  \end{equation}
\end{subequations}
Here and below, the one-point degenerate parabolic operator $\Theta$ with action on $f\in C^{2,1}((0,1)\times (0,T))$ for some $T>0$ reads 
\begin{equation*}
  \Theta f := f_t - 4\rho f_{\rho\rho} - 2f_\rho \biggl(f - \frac{m\rho}{2\pi}\biggr).
\end{equation*}
In view of \cite[Theorem 4.1]{Winkler2010a}, the Dirichlet initial-boundary value problem~\eqref{sys: Theta} admits comparison arguments. 
To detect finite-time blowup, 
they constructed a gradient explosion auxiliary function $g$ satisfying  
\begin{equation*}
  \Theta g \leq 0 
\end{equation*} 
in the sense of distribution,   
where $g$ near the left lateral boundary $\{0\}\times\{t>0\}$ is particularly defined by 
\begin{equation}
  \label{sym: g}
  g := \frac{a\rho}{\rho + \tau^3}, 
  \quad \tau= \varepsilon-\ell t,
  \quad \text{for } \rho\in(0,\rho_0) \text{ and } t\in(0,\varepsilon/\ell),
\end{equation}
with certain positive constants $a>0$, $\varepsilon > 0$, $\ell > 0$ and $\rho_0\in(0,1)$.
Note that the ``stationary'' problem of \eqref{sys: Theta}
\begin{equation*}
  \begin{cases}
  4\rho F_{\rho\rho} + 2FF_\rho = 0, & \rho > 0,\\
  F(0) = 0, 
  \end{cases}
\end{equation*}
has a family of solutions 
\begin{equation*}
  F_\lambda (\rho) := F_0(\rho/\lambda^2) = \frac{4\rho}{\rho + \lambda^2},
  \quad\text{with } 
  F_0(\rho) := \frac{4\rho}{\rho + 1} 
  \quad \text{for } \lambda, \rho > 0,
\end{equation*}
which are mass distribution functions of solutions 
\begin{equation*}
  2F_\lambda^\prime \Bigm|_{\rho = |x|^2} 
  = \frac{8\lambda^2}{(|x|^2 + \lambda^2)^2}
  = \frac{1}{\lambda^2}V_0\bigg(\frac{x}{\lambda}\bigg), 
  \quad x\in\mathbb{R}^2,
\end{equation*}
to the stationary problem~\eqref{sys: V_0}. 
Therefore, $g$ might be thought as a perturbation of $F_{\tau^{3/2}}(\rho)$. 
With the construction~\eqref{sym: g} at hands, 
we have by $a>4$ and the concavity of $g$ with respect to $\rho$,
\begin{equation*}
  \Theta g 
  \leq g_t + \frac{m}{\pi}\cdot \rho g_\rho 
  \leq g_t + \frac{m}{\pi}\cdot g 
  \quad \text{for } (\rho,t)\in (0,1)\times(0,\varepsilon/\ell),
\end{equation*}
which encourages us to modify $g$ with a temporal decay at a slightly strong exponential rate such that $\Theta g \leq 0$ for $(\rho,t)\in (0,1)\times(0,\varepsilon/\ell)$. 

Indeed, we may perturb~\eqref{sym: g} to produce a subsolution over the entire parabolic cylinder.

\begin{lemma}
  \label{le: subsolution of JL}
  For any $m>8\pi$, there exist $\varepsilon>0$ and $\ell>0$ such that the function 
  \begin{equation}
    \um := \frac{a\rho}{\rho+\tau^3}, 
    \quad (\rho,t)\in [0,1]\times[0,\varepsilon/\ell)
  \end{equation}
  complies with 
  \begin{equation}
    \label{eq: estimates of um at right end}
    \sup_{t\in(0,\varepsilon/\ell)} \um(1,t) 
    < \frac{m}{2\pi} 
  \end{equation}  
  and 
  \begin{equation}
    \inf_{\rho\in(0,1)} \um_\rho(\cdot, t) > 0
    \quad \text{for all } t\in[0,\varepsilon/\ell),
  \end{equation}
  as well as 
  \begin{equation}
    \label{eq: Theta um leq 0}
    \Theta \um \leq 0 
    \quad \text{for all }\rho\in(0,1) 
    \text{ and } t\in(0,\varepsilon/\ell),
    \end{equation}
  where 
\begin{equation*} 
\tau = \tau(t) = \varepsilon - \ell t \quad \text{for } t\in(0,\varepsilon/\ell)
\end{equation*}
and 
\begin{equation*}
  a = a(t) = 4\e^\tau \quad \text{for } t\in(0,\varepsilon/\ell).
\end{equation*}
In particular,
\begin{equation*}
  \um(0,t)\equiv 0 
  \quad\text{for all } t\in(0,\varepsilon/\ell),
\end{equation*}
but 
  \begin{equation}
    \label{eq: um approach to 4}
    \um(\rho,t)\to4
    \quad\text{for all } \rho\in(0,1],
  \end{equation}
  as $t\uparrow\varepsilon/\ell$.
\end{lemma}

\begin{proof}
  Fix $\varepsilon\in(0,1)$ and $\ell\in(0,1)$ complying with 
  \begin{equation}
    \label{eq: varepsilon small}
    \e^\varepsilon < \frac{m}{8\pi}
  \end{equation}
  and 
  \begin{equation}
    \label{eq: Vieta theorem}
    \bigg(3\ell+\frac{m\varepsilon}{\pi}\bigg)^2 - 32\ell\leq0.
  \end{equation}
  One may check that choices
  \begin{equation*}
    \varepsilon 
    := \min\bigg\{\frac{1}{2}, \log \frac{m+8\pi}{16\pi}, 16\cdot\Big(3+\frac{m}{\pi}\Big)^{-2}\bigg\}
    \quad\text{and}\quad 
    \ell := \varepsilon
  \end{equation*}
  satisfy \eqref{eq: varepsilon small} and \eqref{eq: Vieta theorem}.
  Note that \eqref{eq: Vieta theorem} ensures that the quadratic inequality 
  \begin{equation}
    \label{eq: quadratic inequality}
    -8\xi^2 + \Big(3\ell+\frac{m\varepsilon}{\pi}\Big)\xi - \ell \leq0
  \end{equation}
  is valid for all $\xi\in\mathbb{R}$.
It follows from \eqref{eq: varepsilon small} that 
\begin{equation}
  \um(1,t) 
  = \frac{a}{1+\tau^3}
  \leq a 
  \leq 4\e^\varepsilon 
  < \frac{m}{2\pi} 
  \quad \text{for all } t\in(0,\varepsilon/\ell).
\end{equation}
This readily warrants~\eqref{eq: estimates of um at right end}.
Evaluating  
\begin{equation*}
  \um_t 
  = \frac{a'}{a}\um 
  - \frac{3\tau^2 \tau'}{\rho + \tau^3}\um
  = -\ell\um 
  + \frac{3\tau^2 \ell}{\rho + \tau^3}\um
\end{equation*}
and 
\begin{equation*}
  \um_\rho = \frac{a\tau^3}{(\rho+\tau^3)^2}
\end{equation*}
and 
\begin{equation*}
  \rho \um_{\rho\rho} 
  = -\frac{2a\tau^3\rho}{(\rho+\tau^3)^3} 
  = -\frac{2\tau^3}{(\rho+\tau^3)^2}\um,
\end{equation*}
we have
\begin{equation*} 
  \begin{aligned}
    \Theta\um 
    &= \um_t - 4\rho \um_{\rho\rho} - 2\um_\rho \bigg(\um - \frac{m\rho}{2\pi}\bigg)\\
    &= \bigg(\frac{2(4-a)\tau^3}{(\rho+\tau^3)^2} 
    + \frac{3\tau^{2}\ell}{\rho+\tau^3} 
    + \frac{m}{\pi}\cdot\frac{\tau^3}{\rho+\tau^3} 
    - \ell\bigg)\um \\
    &\leq \bigg(-8\cdot \frac{(\e^\tau-1)\tau^3}{(\rho+\tau^3)^2} 
    + 3\ell\cdot\frac{\tau^{2}}{\rho+\tau^3} 
    + \frac{m\varepsilon}{\pi}\cdot\frac{\tau^2}{\rho+\tau^3} - \ell\bigg)\um
  \end{aligned} 
\end{equation*}
for all $(\rho,t)\in(0,1)\times(0,\varepsilon/\ell)$.
Therefore, \eqref{eq: Theta um leq 0} follows from the identity 
\begin{equation}
  \label{eq: a basic inequalities}
  \e^\xi = \sum_{i=0}^\infty \frac{\xi^i}{i!}
  \quad \text{for }\xi\in\mathbb{R},
\end{equation}
and \eqref{eq: quadratic inequality}.
\eqref{eq: um approach to 4} is obvious. 
\end{proof}

Now we take the reflection process of \cite[Lemma~3.2]{Mao2024c}.

\begin{lemma}
  \label{le: function comparison}
  If $f$ belongs to $C^0([0,1])\cap C^1((0,1))$ with $f(0)=0$ 
  and 
   \begin{equation*}
     0 < \inf_{\xi\in(0,1)}f'(\xi) \leq \sup_{\xi\in(0,1)} f'(\xi) < \infty,
   \end{equation*}
  then there exists $\theta > 0$ such that for all $\epsilon > 0$, 
  \begin{equation*}
    \frac{(1+\theta)\xi}{\theta + \xi}\cdot(f(1)+\epsilon) \geq f(\xi) 
    \quad \text{for all } \xi\in(0,1). 
  \end{equation*}   
  Moreover, if $f\in C^0([0,1])$ is nonnegative and increasing with 
  \begin{equation*}
    \sup_{\xi\in(0,1)}\frac{f(\xi)}{\xi} < \infty,
  \end{equation*}
  then for any $\epsilon > 0$, there exists $\theta > 0$ such that 
  \begin{equation*}
    \frac{(1+\theta)\xi}{\theta + \xi}\cdot(f(1)+\epsilon) \geq f(\xi) 
    \quad \text{for all } \xi\in(0,1). 
  \end{equation*}  
\end{lemma}

We are in a position to show Proposition~\ref{prop: finite-time blowup in JL system}.

\begin{proof}[Proof of Proposition~\ref{prop: finite-time blowup in JL system}]
  Let the mapping 
  \begin{equation*}
    \um: [0,1]\times[0,\varepsilon/\ell)\mapsto C^{1,0}([0,1]\times[0,\varepsilon/\ell))\cap C^{2,1}((0,1)\times(0,\varepsilon/\ell))
  \end{equation*} 
  be given as in Lemma~\ref{le: subsolution of JL}. 
  Put 
  \begin{equation*}
    \tilde{M}(\rho) := 2\pi\um(\rho,0)\quad \text{for }\rho\in[0,1]
  \end{equation*}
  and 
  \begin{equation*}
    T^\star = \frac{\varepsilon}{\ell}.
  \end{equation*}
  Then \eqref{eq: estimates of um at right end} assures \eqref{eq: boundary conditions}.
  Define the set accordingly
  \begin{equation*}
    \mathcal{S}(m) :=\{u_0: 
    u_0\text{ complies with }\eqref{h: initial data u0}, \eqref{h: initial mass constraint of JL} \text{ and } \eqref{h: initial comparison on u0}.\}. 
  \end{equation*}
  Thanks to Lemma~\ref{le: function comparison}, $\mathcal{S}(m)$ is nonempty. 
  For any $u_0\in \mathcal{S}(m)$, the system~\eqref{sys: JL system} admits a uniquely determined classical solution $(u,v)$ with the maximal existence time $T_{\max} = T_{\max}(u_0)\in(0,\infty]$ from the function class~\eqref{eq: function class of JL},   
  complying with \eqref{eq: extensibility} if $T_{\max} < \infty$,
  according to \cite[Theorem~1.3]{Cieslak2008}.
  We claim 
  \begin{equation*}
  T_{\max} \leq T^\star.  
  \end{equation*}

  Indeed, if $T_{\max} > T^\star$, we have by~\eqref{eq: function class of JL}
  \begin{equation}
  \label{eq: pointwise estimate at t=varepsilon/ell}
    \sup_{\rho\in(0,1)}\rho^{-1}M(\rho,T^\star)
    \leq \frac{1}{2}\cdot \|u(\cdot,T^\star)\|_{L^\infty(\Omega)} 
    < \infty
  \end{equation}
  and 
  \begin{equation*}
    M\geq \um\quad\text{for all }(\rho,t)\in(0,1)\times(0,T^\star),
  \end{equation*} 
  by comparison arguments~\cite[Theorem 4.1]{Winkler2010a}.
  However, \eqref{eq: um approach to 4} implies 
  \begin{equation*}
    M(\rho,T^\star)\geq 4 \quad\text{for all } \rho\in(0,1],
  \end{equation*}
  which is incompatible with the pointwise upper estimate~\eqref{eq: pointwise estimate at t=varepsilon/ell}.
\end{proof}

\section{Finite-time blowup}
\label{sec: finite-time of indirect chemotaxis}

This section is devoted to the proof of Theorem~\ref{thm: finite-time blowup}, based on ideas recorded in Section~\ref{sec: motivations}.

\subsection{Preliminaries}
\label{section preliminary}

Let us begin with local solvability, uniqueness and extensibility of classical solutions to the system \eqref{sys: ks isp pep},
which have been established in~\cite[Section~2]{Tao2025}.

\begin{proposition}
  \label{prop: local existence and uniqueness}
Let $\Omega$ be the unit ball in $\mathbb{R}^4$ and assume~\eqref{h: initial data}. Then there exist $T_{\max } \in(0, \infty]$ and a uniquely determined triplet of radially symmetric functions
\begin{equation}
\begin{cases}
u \in C^0(\bar{\Omega} \times[0, T_{\max })) 
\cap C^{2,1}(\bar{\Omega} \times(0, T_{\max })), \\
v \in C^{2,0}(\bar{\Omega} \times(0, T_{\max })) \quad \text { and } \\
w \in C^0(\bar{\Omega} \times[0, T_{\max })) 
\cap C^{2,1}(\bar{\Omega} \times(0, T_{\max }))
\end{cases}
\end{equation}
with the properties that $u \geq 0$ and $w \geq 0$ in $\Omega \times(0, T_{\max })$, that $\int_{\Omega} v(\cdot, t)=0$ for all $t \in(0, T_{\max })$, that \eqref{sys: ks isp pep} is satisfied in the classical sense in $\Omega \times(0, T_{\max })$, and that
\begin{equation}
  \label{eq: extensibility principle}
\text {if } T_{\max }<\infty, \quad \text {then }\limsup _{t \uparrow T_{\max }}\|u(\cdot, t)\|_{L^{\infty}(\Omega)}=\infty.
\end{equation}
Moreover, 
\begin{equation} 
  \label{eq: mass conservation}
\int_{\Omega} u(\cdot, t) \dd x = \int_\Omega u_0 \dd x 
\quad \text {for all } t \in(0, T_{\max }). 
\end{equation}
\end{proposition}

Following the precedents~\cite{Jaeger1992} and recent development~\cite{Tao2025} due to Tao and Winkler, we shall analyze a cooperative parabolic system satisfied by cumulated densities to detect radial blowup in the J\"ager-Luckhaus type chemotaxis model~\eqref{sys: ks isp pep}.

\begin{lemma} 
  \label{le: cumulated densities}
The mass distribution substitution 
\begin{equation}
  \label{sym: mathcal L}
  \mathcal{L}f(s) := \frac{1}{2\pi^2}\int_{B_{s^{1/4}}}f\dd x 
  = \int_0^{s^{1/4}}fr^3\dd r 
  \quad \text{for } f\in C^0(\Omega) \text{ and } s\in(0,1),
\end{equation}
coverts the triplet $(u,v,w)$ of radial functions from Proposition~\ref{prop: local existence and uniqueness} to $(U,W) := (\mathcal{L}u,\mathcal{L}w)$ for $(s,t)\in[0,1]\times[0,T_{\max})$. 
The function couple $(U,W)$ satisfies the cooperative system 
\begin{equation}
\label{sys: mathcal P Q}
    \mathcal{P}(U,W) = 0 
    \quad\text{and}\quad  
    \mathcal{Q}(U,W) = 0
    \quad\text{for all } (s,t)\in(0,1)\times(0,T_{\max})
\end{equation}
with boundary conditions 
\begin{equation} 
\label{sys: left boundary conditions}
    U(0,t) = W(0,t) = 0
    \quad \text{for all } t\in(0,T_{\max})
\end{equation}
and 
\begin{equation}
\label{sys: right boundary conditions}
  U(1,t) = \frac{m}{2\pi^2}
  \quad\text{and}\quad 
  W(1,t) = W_0(1)\e^{-\delta t} + \frac{m\delta^{-1}}{2\pi^2}(1-\e^{-\delta t})
\end{equation}
for all $t\in(0,T_{\max})$,
where the initial datum is given by 
\begin{equation} 
  \label{sym: U_0 W_0}
(U_0,W_0) := (\mathcal{L}u_0,\mathcal{L}w_0)
\quad \text{for } s\in(0,1).
\end{equation}
Here and below, for $T>0$ and functions $\phi$ and $\psi$ from $C^1([0,1]\times[0,T))$ which satisfy $\phi_s\geq0$ on $(0,1)\times(0,T)$ and are such that $\phi(\cdot,t)\in W^{2,\infty}_{\loc}((0,1))$ and $\psi(\cdot,t)\in W^{2,\infty}_{\loc}((0,1))$ for all $t\in(0,T)$, define 
\begin{equation}
  \mathcal{P}(\phi,\psi) 
  = \phi_t - 16s^{3/2} \phi_{ss} - 4\phi_s\bigg(\psi - \frac{\mu_\ast s}{4}\bigg)
\end{equation}
with 
\begin{equation}
  \label{sym: mu ast}
  \mu_\ast(t) %
  = 4\psi(1,t)
  \quad\text{for } t \in (0,T_{\max}),
\end{equation}
and 
\begin{equation}
  \label{sym: mathcal Q}
  \mathcal{Q}(\phi,\psi) = \psi_t - 16s^{3/2}\psi_{ss} + \delta\psi - \phi
\end{equation}
as well as 
\begin{equation}
  \label{sym: mathcal P mu ast}
  \mP(\phi,\psi) 
  = \phi_t - 16s^{3/2} \phi_{ss} - 4\phi_s\bigg(\psi - \frac{\mu^\star s}{4}\bigg)
\end{equation}
for $t\in(0,T)$ and a.e. $s\in(0,1)$, 
where $\mu^\ast > 0$
is a constant. 
\end{lemma}

As a crucial aspect of this substitution~\eqref{sym: mathcal L}, the cooperative system $\mP(\phi,\psi) = 0$ and $\mathcal{Q}(\phi,\psi) = 0$ admits comparison arguments~\cite[Lemma~3.2]{Tao2025}.

\begin{lemma}
  \label{le: comparison}
Let $\mu^{\star}>0$, and suppose that for some $T>0$, 
the functions $\uu, \ou, \uw$ and $\ow$ belong to $C^1([0, 1]\times[0, T))$ 
and have the properties that 
$\{\uu(\cdot, t), \ou(\cdot, t), \uw(\cdot, t), \ow(\cdot, t)\} \subset W_{\loc}^{2, \infty}((0, 1))$ for all $t \in(0, T)$, 
and that $\uu_s$ and $\ou_s$ are nonnegative on $(0, 1) \times(0, T)$. 
If with $\mP$ and $\mathcal{Q}$ as in \eqref{sym: mathcal P mu ast} and \eqref{sym: mathcal Q} one has 
\begin{equation} 
  \label{eq: mP uu uw leq mP ou ow}
\mP(\uu, \uw)(s, t) 
\leq 0 \leq 
\mP(\ou, \ow)(s, t)
\quad \text{for all } t \in(0, T) \text{ and a.e. } s \in(0, 1) 
\end{equation} 
as well as
\begin{equation} 
  \label{eq: mQ uu uw leq mQ ou ow}
  \mathcal{Q}(\uu, \uw)(s, t) \leq 0 \leq \mathcal{Q}(\ou,\ow)(s, t) 
  \quad \text{for all } t \in(0, T) \text{ and a.e. } s \in(0, 1),
\end{equation} 
and if furthermore
\begin{equation*}
  \uu(0, t) \leq \ou(0, t), \quad \uu(1, t) \leq \ou(1, t), \quad \uw(0, t) \leq \ow(0, t) \quad \text {and} \quad \uw(1, t) \leq \ow(1, t)
\end{equation*}  
for all $t \in[0, T)$,
as well as  
  \begin{equation*}
  \uu(s, 0) \leq \ou(s, 0) \quad \text {and} \quad \uw(s, 0) \leq \ow(s, 0) 
  \quad \text {for all } s \in[0, 1],
  \end{equation*}  
then  
  \begin{equation*}
  \uu(s, t) \leq \ou(s, t) 
  \quad \text {for all } s \in[0, 1] \text { and } t \in[0, T) .
  \end{equation*}
\end{lemma}

Naturally, we call a function pair $(\uu,\uw)$ the subsolution, 
if both \eqref{eq: mP uu uw leq mP ou ow} and \eqref{eq: mQ uu uw leq mQ ou ow} hold.

\subsection{Stationary solutions}

As the special solution~\eqref{sym: V_0} to the stationary problem~\eqref{sys: V_0} is of importance to construct blowup solutions for classical chemotaxis models, 
we shall recall the steady states in indirect chemotaxis \cite[Proposition~9.1]{Fujie2017}.
The system 
\begin{align*}
  \begin{cases}
    0=\e^Z - X & \text { in } \mathbb{R}^4, \\ 
    0=\Delta Z + Y & \text { in } \mathbb{R}^4, \\ 
    0=\Delta Y + X & \text { in } \mathbb{R}^4,
  \end{cases}
\end{align*}
has a family of radially symmetric solutions 
\begin{equation*}
  (X_\lambda, Z_\lambda, Y_\lambda) 
  := (\lambda^{4}X_0(\lambda x), 
  \log X_\lambda, 
  \lambda^{2}Y_0(\lambda x))
  \quad\text{for } \lambda > 0, 
\end{equation*}
where 
\begin{equation}
  X_0(x) := \frac{2^7\cdot3}{(1+|x|^2)^4}
  \quad\text{and}\quad 
  Y_0(x) := \frac{16(2+|x|^2)}{(1+|x|^2)^2},
  \quad x\in\mathbb{R}^4.
\end{equation}
We calculate the mass distribution transformation acting on $X_0$ and $Y_0$ accordingly 
\begin{equation*}
  \begin{aligned}
    \mathcal{L}X_0 
    &= \int_0^{s^{1/4}}\frac{2^7\cdot3r^3\dd r}{(1+r^2)^4} 
    = \int_0^{s^{1/2}}\frac{2^6\cdot3\xi\dd\xi}{(1+\xi)^4}\\
    &= - \int_0^{s^{1/2}} 2^6\xi\dd(1+\xi)^{-3} 
    = - \frac{2^6s^{1/2}}{(1+s^{1/2})^3} 
    + \int_0^{s^{1/2}}\frac{2^6\dd\xi}{(1+\xi)^3}\\
    &= - \frac{2^6s^{1/2}}{(1+s^{1/2})^3} 
    + 2^5 - \frac{2^5}{(1+s^{1/2})^2}
    = \frac{2^5(s^{3/2} + 3s)}{(1+s^{1/2})^3}
  \end{aligned}
\end{equation*}
and 
\begin{equation*} 
\begin{aligned}
  \mathcal{L}Y_0 
  &= 16\int_0^{s^{1/4}}\frac{(2+r^2)r^3\dd r}{(1+r^2)^2} 
  = 8\int_0^{s^{1/2}}\frac{(2+\xi)\xi}{(1+\xi)^2}\dd \xi \\
  &= 8\int_0^{s^{1/2}}\dd\bigg(\frac{\xi^2}{1+\xi}\bigg) 
  = \frac{8s}{1+s^{1/2}}
\end{aligned} 
\end{equation*}
for all $s\geq0$.
Then
\begin{equation}  
  \label{sym: mathcal ulambda}
\mathcal{L}X_\lambda 
= \mathcal{L}X_0(\lambda^{4} s) 
= \frac{2^5(s^{3/2}+3\lambda^{-2}s)}{(s^{1/2} + \lambda^{-2})^3}
\quad\text{for } s\geq0 \text{ and } \lambda > 0,
\end{equation} 
and 
\begin{equation} 
  \label{sym: mathcal wlambda}
\mathcal{L}Y_\lambda 
= \lambda^{-2}\mathcal{L}Y_0(\lambda^{4} s)
= \frac{8s}{s^{1/2}+\lambda^{-2}}
\quad\text{for } s\geq0 \text{ and } \lambda > 0.
\end{equation}
One may check that for each $\lambda>0$, $(\mathcal{L}X_\lambda, \mathcal{L}Y_\lambda)$ solves 
\begin{equation*}
  \begin{cases}
    0 = 16s^{3/2}\Phi_{ss} + 4\Phi_s\Psi, & s > 0,\\
    0 = 16s^{3/2}\Psi_{ss} + \Phi, & s > 0.
  \end{cases}
\end{equation*}
We shall construct a family of function pairs $(\uphi,\upsi)$ by perturbing the special functions \eqref{sym: mathcal ulambda} and \eqref{sym: mathcal wlambda} accordingly 
\begin{equation*}
  \uphi(s,t) = \frac{a(s^{3/2}+ 3\lambda^2s)}{(s^{1/2} + \lambda^2)^3}
  \quad \text{for } s\in[0,1] \text{ and } 0<\lambda\ll1
\end{equation*}
and 
\begin{equation*}
  \upsi(s,t) = \frac{bs}{s^{1/2}+\lambda^2} 
  \quad \text{for } s\in[0,1] \text{ and } 0<\lambda\ll1,
\end{equation*}
where $a$, $b$ and $\lambda$ are time-dependent nonnegative functions to be specified below such that $\mP(\uphi,\upsi) \leq0$ and $\mathcal{Q}(\uphi,\upsi) \leq 0$ for  $s\in[0,1]$ and $0<\lambda\ll1$.

\subsection{Construction of exploding subsolutions} 

Before proceeding to construct subsolutions with finite-time gradient blowup,
we shall prepare some parameters.

\begin{lemma}
  \label{le: parameters}
Let $\delta > 0$.
  For any $m>64\pi^2$, $\kappa > 0$ and $\mu^\star > 0$, one can find $\varepsilon\in(0,1)$, $\ell > 0$, $\gamma > 0$ with properties that 
\begin{equation}
  \label{eq: ell geq delta}
  \ell \geq \delta,
\end{equation}
that 
\begin{equation}
\label{eq: a quadratic inequality}
  - 2^5\xi^2 
  + (\mu^\star\varepsilon + 6\ell) \xi 
  - (\gamma+1)\ell 
  \leq 0 
  \quad\text{for all } \xi\in\mathbb{R},
\end{equation}
that  
\begin{equation}
  \label{eq: gamma2 geq 3ell}
  2\gamma^2-3\ell \geq 0,
\end{equation}
that 
\begin{equation}
  \label{eq: e^varepsilon < something kappa}
  \e^{\varepsilon} < \frac{1+2\kappa}{1+\kappa}
\end{equation}
that
\begin{equation}
\label{eq: small varepsilon}
  \e^{(\gamma+1)\varepsilon}(1+3\varepsilon^3) < \frac{m}{64\pi^2},
\end{equation}
and that 
\begin{equation}
  \label{eq: small varepsilon again}
  g(t) := \e^{(\ell-\delta)t}\big(1+(\varepsilon-\ell t)^3\big) \geq 1 + \varepsilon^3
  \quad\text{for all } t\in(0,\varepsilon/\ell).
\end{equation}
\end{lemma}

\begin{proof}
  Put
  \begin{equation}  
    \label{sym: ell}
  \ell = 2\delta
  \end{equation} 
  and 
  \begin{equation}
    \label{sym: gamma}
    \gamma = \max\bigg\{
      3^{1/2}\delta^{1/2}, \frac{(\mu^\star + 12\delta)^2}{2^{8}\delta}
    \bigg\}.
  \end{equation}
  Thanks to $m>64\pi^2$, the transcendental equation 
    \begin{equation*}
        \e^{(\gamma+1)\xi}(1+3\xi^3) = \frac{m}{64\pi^2} 
        \quad\text{for } \xi \in \mathbb R
  \end{equation*}
  has a unique positive root denoted by $\varepsilon_0$.
  Put 
  \begin{equation}
    \label{sym: varepsilon}
    \varepsilon 
    = \min\bigg\{\frac{1}{3}, 
    \frac{\varepsilon_0}2,
    \frac{1}{2}\ln\frac{1+2\kappa}{1+\kappa}
    \bigg\}.
  \end{equation}
  Then our choices \eqref{sym: ell}, \eqref{sym: gamma} and \eqref{sym: varepsilon} clearly satisfy properties~\eqref{eq: ell geq delta}--\eqref{eq: small varepsilon}. 
  Since 
  \begin{equation*}
    \begin{aligned} 
    g' &= (\ell-\delta)\e^{(\ell-\delta)t}\big(1+(\varepsilon-\ell t)^3\big) 
    - 3\ell\e^{(\ell-\delta)t}(\varepsilon-\ell t)^2 \\
    & = \delta \e^{\delta t}\big(1+(\varepsilon-\ell t)^3 - 6(\varepsilon-\ell t)^2\big) 
    \geq \delta \e^{\delta t}(1-2/3)
    \geq0
    \end{aligned}
  \end{equation*}
  for all $t\in(0,\varepsilon/\ell)$ due to \eqref{sym: ell} and \eqref{sym: varepsilon}, 
  we have $g(t)\geq g(0) = 1 + \varepsilon^3$ and thus get \eqref{eq: small varepsilon again}.
\end{proof}

We may give the family of subsolutions.

\begin{definition}
  \label{def: subsolution}
  Let $\delta>0$. For any $m>64\pi^2$, $\kappa > 0$ and $\mu^\star > 0$,
  take $\varepsilon$, $\ell$ and $\gamma$ from Lemma~\ref{le: parameters}. 
  Put  
\begin{equation}
  \tau(t) = \varepsilon-\ell t,
  \quad t\in(0,\varepsilon/\ell)
\end{equation}
and 
\begin{equation}
  a(t) = 2^5\e^{(\gamma+1)\tau},
  \quad t\in(0,\varepsilon/\ell),
\end{equation}
as well as 
\begin{equation}
  b(t) = 8\e^\tau,
  \quad t\in(0,\varepsilon/\ell).
\end{equation}
  Then one obtains elements $\uu$ and $\uw$ of $C^{2,1}([0,1]\times[0,\varepsilon/\ell))$ defined by 
  \begin{equation}
    \label{sym: uu}
    \uu(s,t) := \frac{a(s^{3/2}+ 3\tau^3s)}{(s^{1/2} + \tau^3)^3}
    \quad \text{for } (s,t)\in[0,1]\times[0,\varepsilon/\ell)
  \end{equation}
  and 
  \begin{equation}
    \label{sym: uw}
    \uw(s,t) := \frac{bs}{s^{1/2} + \tau^3}
    \quad \text{for } (s,t)\in[0,1]\times[0,\varepsilon/\ell).
  \end{equation}
\end{definition}

We check that $\uw$ is one component of the candidate for subsolutions.

\begin{lemma}
  \label{le: Q(uu,uw) leq 0}
  The function pair $(\uu,\uw)$ from Definition~\ref{def: subsolution} has the properties that for each $t\in[0,\varepsilon/\ell)$, 
  the mapping $\uw(s,\cdot): [0,1]\mapsto \mathbb{R}$ is nonnegative, increasing and concave, that 
  \begin{equation}
  \label{eq: uw(0,t)=0}
    \uw(0,t) = 0 \quad\text{for all } t\in[0,\varepsilon/\ell),
  \end{equation}
  that 
\begin{equation}
\label{eq: estimate of uw at right end}
  \sup_{t\in(0,\varepsilon/\ell)}\uw(1,t) < 8\cdot\frac{1+2\kappa}{1+\kappa},
\end{equation}
and that
\begin{equation}
  \label{eq: Q(uu,uw) leq 0}
  \mathcal{Q}(\uu,\uw)\leq0\quad \text{for all } (s,t)\in(0,1)\times(0,\varepsilon/\ell).
\end{equation}
\end{lemma}

\begin{proof}
\eqref{eq: uw(0,t)=0} is obvious. 
Since
\begin{equation}
\label{eq: estimate of uw(1,t)}
  \uw(1,t)
  = \frac{b}{1 + \tau^3}
  \leq b
  \leq 8\e^\varepsilon
  \quad \text{for all }t\in(0,\varepsilon/\ell),
\end{equation}
\eqref{eq: estimate of uw at right end} follows from \eqref{eq: e^varepsilon < something kappa}.
  We compute 
\begin{equation*}
  \uw_s 
  = \frac{b}{s^{1/2} + \tau^3} - \frac{1}{2}\cdot\frac{bs^{1/2}}{(s^{1/2} + \tau^3)^2}
  = \frac{b}{(s^{1/2} + \tau^3)^2}\bigg(\tau^3 + \frac{s^{1/2}}{2}\bigg)
\end{equation*}
and 
\begin{equation*}
  \uw_{ss} 
  = \frac{b}{4}\cdot\frac{s^{-1/2}}{(s^{1/2} + \tau^3)^2}
   - \frac{bs^{-1/2}}{(s^{1/2} + \tau^3)^3} 
   \cdot \bigg(\tau^3 + \frac{s^{1/2}}{2}\bigg)\\
  = -\frac{b}{4}\cdot\frac{1+3\tau^3s^{-1/2}}{(s^{1/2} + \tau^3)^3}
\end{equation*}
and 
\begin{equation*}
  \uw_t 
  = \frac{b's}{s^{1/2} + \tau^3} - \frac{3b\tau^2\tau's}{(s^{1/2} + \tau^3)^2}
  = \bigg(\frac{b'}{b} + \frac{3\tau^2\ell}{s^{1/2} + \tau^3}\bigg)W.
\end{equation*}
So in view of 
\begin{equation*}
  s^{3/2}\uw_{ss} = - \frac{1}{4}\cdot\frac{s^{1/2}+3\tau^3}{(s^{1/2} + \tau^3)^2} \cdot \uw
  \quad\text{and}\quad 
  \uu = \frac{a}{b}\cdot\frac{s^{1/2} + 3\tau^3}{(s^{1/2} + \tau^3)^2}\cdot \uw,
\end{equation*}
we get 
\begin{equation*} 
\begin{aligned}
  \mathcal{Q}(\uu,\uw) 
  &= \uw_t - 16s^{3/2}\uw_{ss} + \delta\uw - \uu\\
  &= \bigg(
    \frac{b'}{b} 
  + \frac{3\tau^2\ell}{s^{1/2} + \tau^3} 
  + \delta 
  + \Big(4-\frac{a}{b}\Big)
    \cdot \frac{s^{1/2}+3\tau^3}{(s^{1/2} + \tau^3)^2}
    \bigg) \cdot \uw\\ 
  &= \bigg(\frac{b'}{b} + \delta \bigg)\cdot \uw  
  + \frac{4-a/b + 3\ell\tau^2}{s^{1/2} + \tau^3}\cdot \uw 
  + \frac{2(4-a/b)\tau^3}{(s^{1/2} + \tau^3)^2} \cdot \uw\\
  &= (\delta - \ell)\uw  
  + \frac{4(1 - \e^{\gamma\tau}) + 3\ell\tau^2}{s^{1/2} + \tau^3} \cdot \uw 
  + \frac{8(1-\e^{\gamma\tau})\tau^3}{(s^{1/2} + \tau^3)^2}\cdot \uw
\end{aligned} 
\end{equation*}
for all $(s,t)\in(0,1)\times(0,\varepsilon/\ell)$,
where we have used the notations $a$ and $b$ in Definition~\ref{def: subsolution}.
Using the identity~\eqref{eq: a basic inequalities},
we estimate 
\begin{equation*}
  \mathcal{Q}(\uu,\uw)\leq (\delta -\ell)\uw 
  + \frac{(3\ell - 2\gamma^2)\tau^2}{s^{1/2} + \tau^3}\cdot\uw 
  \quad \text{for all } (s,t)\in(0,1)\times(0,\varepsilon/\ell).
\end{equation*}
Thus \eqref{eq: Q(uu,uw) leq 0} is verified according to \eqref{eq: ell geq delta} and \eqref{eq: gamma2 geq 3ell}.
\end{proof}

We check that $\uu$ is also one component of the candidate for subsolutions.

\begin{lemma}
  The function pair $(\uu,\uw)$ from Definition~\ref{def: subsolution} has the properties that for each $t\in[0,\varepsilon/\ell)$, 
  the mapping $\uu(s,\cdot): [0,1]\mapsto \mathbb{R}$ is nonnegative, increasing and concave, that 
  \begin{equation}
  \label{eq: uu(0,t)=0}
    \uu(0,t) \equiv 0 \quad\text{for all } t\in[0,\varepsilon/\ell),
  \end{equation}
  that 
  \begin{equation}
  \label{eq: uu approaches to 32}
    \uu(s,t)\to 2^5 \quad \text{for all } s\in(0,1], 
  \end{equation}
  as $t\uparrow\varepsilon/\ell$,
  that 
\begin{equation}
\label{eq: estimates of uu at right end}
  \sup_{t\in(0,\varepsilon/\ell)}\uu(1,t) < \frac{m}{2\pi^2},
\end{equation}
  and that
  \begin{equation}
  \label{eq: mathcal P(uu,uw) leq 0}
    \mP(\uu,\uw)\leq0\quad \text{for all } (s,t)\in(0,1)\times(0,\varepsilon/\ell).
  \end{equation}
\end{lemma}

\begin{proof}
\eqref{eq: uu(0,t)=0} and \eqref{eq: uu approaches to 32} are obvious.
It follows from \eqref{eq: small varepsilon} that 
\begin{equation*}
\begin{aligned} 
  \uu(1,t) 
  &= \frac{a(1+3\tau^3)}{(1+\tau^3)^3}
  \leq a(1+3\tau^3) \\ 
  &\leq 2^5\e^{(\gamma+1)\varepsilon}(1+3\varepsilon^3)
  < \frac{m}{2\pi^2} 
  \quad \text{for all } t\in(0,\varepsilon/\ell),
  \end{aligned}
\end{equation*}
which readily implies \eqref{eq: estimates of uu at right end}.
  We calculate
  \begin{equation*}
    \uu_s 
    = \frac{3a}{(s^{1/2} + \tau^3)^3}
    \cdot\bigg(\frac{s^{1/2}}{2} + \tau^3\bigg)
    - \frac{3a(s^{3/2}+3\tau^3s)}{(s^{1/2} + \tau^3)^4}\cdot\frac{s^{-1/2}}{2}
    = \frac{3a\tau^6}{(s^{1/2} + \tau^3)^4} 
  \end{equation*} 
  and hence 
  \begin{equation*}
    \uu_{ss} = - \frac{6a\tau^6s^{-1/2}}{(s^{1/2} + \tau^3)^5}
  \end{equation*}
as well as  
\begin{align*}
  \uu_t &= \frac{a'(s^{3/2}+ 3\tau^3s)}{(s^{1/2} + \tau^3)^3} 
  + \frac{9a\tau^2\tau's}{(s^{1/2} + \tau^3)^3}
  - \frac{9a(s^{3/2}+ 3\tau^3s)\tau^2\tau'}{(s^{1/2} + \tau^3)^4}\\
  &= \frac{a'(s^{3/2}+ 3\tau^3s)}{(s^{1/2} + \tau^3)^3} 
  + \frac{18a\ell\tau^5s}{(s^{1/2} + \tau^3)^4}.
\end{align*}
Therefore, we have 
\begin{align*}
  \mP(\uu,\uw) 
  & = \uu_t 
  - 16s^{3/2} \uu_{ss} 
  - 4\uu_s\bigg(\uw - \frac{\mu^\star s}{4}\bigg)\\
  &= \frac{a'(s^{3/2}+ 3\tau^3s)}{(s^{1/2} + \tau^3)^3} 
  + \frac{18a\ell\tau^5s}{(s^{1/2} + \tau^3)^4}\\ 
  &\quad + \frac{2^5\cdot3a\tau^6s}{(s^{1/2} + \tau^3)^5} 
  - \frac{12ab\tau^6s}{(s^{1/2} + \tau^3)^5}
  + \frac{3a\mu^\star\tau^6 s}{(s^{1/2} + \tau^3)^4}\\
  &= \frac{3a\tau^3s}{(s^{1/2} + \tau^3)^3}\cdot I,
\end{align*}
where 
\begin{align*}
  I &:= \frac{(2^5-4b)\tau^3}{(s^{1/2} + \tau^3)^2} 
  + \frac{\mu^\star \tau^3 + 6\ell\tau^2}{s^{1/2} + \tau^3} + \frac{a's^{1/2}}{3a\tau^3} + \frac{a'}{a}\\
  &\:= \frac{2^5(1-\e^\tau)\tau^3}{(s^{1/2} + \tau^3)^2} 
  + \frac{\mu^\star \tau^3 + 6\ell \tau^2}{s^{1/2} + \tau^3} 
  - \frac{s^{1/2}}{3\tau^3}\cdot (\gamma+1)\ell 
  - (\gamma+1)\ell 
\end{align*}
for all $(s,t)\in(0,1)\times(0,\varepsilon/\ell)$.
The identity~\eqref{eq: a basic inequalities} implies 
\begin{equation*}
I \leq - 2^5\cdot\frac{\tau^4}{(s^{1/2} + \tau^3)^2} 
  + (\mu^\star\varepsilon+6\ell) \cdot \frac{\tau^2}{s^{1/2} + \tau^3}
  - (\gamma+1)\ell
\end{equation*}
for all $(s,t)\in(0,1)\times(0,\varepsilon/\ell)$.
It follows from \eqref{eq: a quadratic inequality} that $I\leq0$ for all $(s,t)\in(0,1)\times(0,\varepsilon/\ell)$ and thus \eqref{eq: mathcal P(uu,uw) leq 0} holds.
\end{proof}

\subsection{Proof of main theorem}

We are in a position to show Theorem~\ref{thm: finite-time blowup}.

\begin{proof}[Proof of Theorem~\ref{thm: finite-time blowup}]
  Put 
  \begin{equation}
    \label{eq: choice of mu star}
    \mu^\star := 2^5\bigg(1+\kappa+\frac{1}{\kappa}\bigg) + \frac{2m}{\pi^2\delta}.
  \end{equation}
  Take the function pair 
  \begin{equation*}
    (\uu,\uw) \in C^{2,1}([0,1]\times[0,T))\times C^{2,1}([0,1]\times[0,T))
  \end{equation*} 
  with 
  \begin{equation*}
    T := \varepsilon/\ell > 0,
  \end{equation*}
  from Definition~\ref{def: subsolution}. 
  Put 
  \begin{equation*}
    M^{(u)}(\cdot) := 2\pi^2\uu(\cdot,0)
    \quad\text{and}\quad 
    M^{(w)}(\cdot) := 2\pi^2\uw(\cdot,0)
    \quad \text{for } s\in(0,1).
  \end{equation*}
  Clearly, \eqref{eq: vanish at zero} is valid according to \eqref{eq: uw(0,t)=0} and \eqref{eq: uu(0,t)=0}, 
  and \eqref{eq: estimates at right end} follows from \eqref{eq: estimate of uw at right end} and \eqref{eq: estimates of uu at right end}.
  Let $(u_0,w_0)$ be a function pair complying with \eqref{h: initial data}--\eqref{h: initial comparison}. 
  (Note that \eqref{eq: estimates at right end} warrants the existence of such initial data according to Lemma~\ref{le: function comparison}.)
  We claim that the corresponding classical solution $(u,v,w)$ of \eqref{sys: ks isp pep} 
  given by Proposition~\ref{prop: local existence and uniqueness} blows up in finite time $T_{\max} \leq T$.

  Suppose by absurdum that $T_{\max} > T$. Then we have 
  \begin{equation}
    \label{eq: u is bounded above at T}
    \|u(\cdot, T)\|_{L^\infty(\Omega)} < \infty
  \end{equation}
  by Proposition~\ref{prop: local existence and uniqueness}.
  Let the function pair $(U,W) = (\mathcal{L}u,\mathcal{L}w)$ for $(s,t)\in[0,1]\times[0,T_{\max})$ be from Lemma~\ref{le: cumulated densities}.
  \eqref{eq: u is bounded above at T} gives 
  \begin{equation}
    \label{eq: pointwise estimate of U at T}
    U(s,T) \leq \frac{s}{4}\cdot \|u(\cdot, T)\|_{L^\infty(\Omega)} 
    \quad\text{for all } s\in(0,1).
  \end{equation}
  Now we aim to apply maxmium principle recorded in Lemma~\ref{le: comparison} to compare $(U,W)$ with $(\uu,\uw)$ over $[0,1]\times[0,T)$.
  In view of \eqref{sym: mu}, \eqref{sym: mu ast}, \eqref{sys: right boundary conditions}, \eqref{h: initial mass constraint} and \eqref{eq: choice of mu star}, with choices of $(\phi,\psi) = (U,W)$, we have 
  \begin{equation*}
    \begin{split} 
    \mu_\ast(t) 
    = \mu(t) 
    &= 4W(1,t) \\
    &= 4W_0(1)\e^{-\delta t} + \frac{2m\delta^{-1}}{\pi^2}(1-\e^{-\delta t})\\
    &\leq \frac{2}{\pi^2}\int_\Omega w_0\dd x + \frac{2m\delta^{-1}}{\pi^2}\\
    &\leq 
    \mu^\star 
    \quad \text{for all } t\in(0,T),
    \end{split}
  \end{equation*} 
  and thus thanks to \eqref{sys: mathcal P Q} and \eqref{eq: mathcal P(uu,uw) leq 0},
  \begin{equation}
    \label{eq: mathcal P mu U W geq uu uw}
    \mP(U,W) 
    \geq \mathcal{P}(U,W) 
    = 0 
    \geq \mP(\uu,\uw) 
    \quad \text{for all } (s,t)\in(0,1)\times(0,T).
  \end{equation}
As direct consequences of \eqref{sys: mathcal P Q} and \eqref{eq: Q(uu,uw) leq 0},
\begin{equation}
  \mathcal{Q}(U,W) 
  = 0 
  \geq \mathcal{Q}(\uu,\uw)
  \quad \text{for all } (s,t)\in(0,1)\times(0,T).
\end{equation}
Note that \eqref{h: initial comparison} guarantees 
\begin{equation}
  U_0(\cdot) \geq \uu(\cdot,0) 
  \quad\text{and}\quad 
  W_0(\cdot) \geq \uw(\cdot,0)
  \quad\text{for all } s\in[0,1],
\end{equation}
where $(U_0,W_0)$ is defined in \eqref{sym: U_0 W_0}.
It follows from \eqref{sys: right boundary conditions} and \eqref{eq: estimates of uu at right end} that 
\begin{equation}
  U(1,t) 
  \equiv \frac{m}{2\pi^2} 
  > \uu(1,t) 
  \quad\text{for all } t\in(0,T).
\end{equation}
Using \eqref{sys: right boundary conditions} and \eqref{h: initial comparison}, as well as \eqref{sym: uw} and \eqref{eq: small varepsilon again},
we compute  
\begin{equation}
  \label{eq: W(1,t) geq uw(1,t)}
  \begin{split}
    W(1,t) 
    &= W_0(1)\e^{-\delta t} + \frac{m\delta^{-1}}{2\pi^2}(1-\e^{-\delta t})\\
    &\geq W_0(1)\e^{-\delta t} 
    \geq \frac{\e^{-\delta t}}{2\pi^2}\cdot M^{(w)}(1) \\
    &= \uw(1,0)\e^{-\delta t}
    = \frac{b(0)\e^{-\delta t}}{1+\varepsilon^3}
    =\frac{8\e^{\varepsilon-\delta t}}{1+\varepsilon^3}\\ 
    &\geq \frac{8\e^{\varepsilon-\ell t}}{1+(\varepsilon-\ell t)^3}
    = \frac{b}{1+\tau^3}
    = \uw(1,t) 
    \quad \text{for all } t\in(0,T).
  \end{split}
\end{equation}
Since $U$ and $W$, as well as $\uu$ and $\uw$, vanish at the left lateral boundary according to \eqref{sys: left boundary conditions}, \eqref{eq: uw(0,t)=0} and \eqref{eq: uu(0,t)=0} for all $t\in(0,T)$, 
gathering \eqref{eq: mathcal P mu U W geq uu uw}--\eqref{eq: W(1,t) geq uw(1,t)} leads us to employ Lemma~\ref{le: comparison} with the choices $(\ou,\ow) = (U,W)$ to obtain 
\begin{equation}
  U \geq \uu 
  \quad \text{for all } (s,t)\in[0,1]\times[0,T).
\end{equation}
Therefore, it follows from \eqref{eq: uu approaches to 32} that 
\begin{equation}
  \label{eq: chemotactic collapse ?}
  U(\cdot,T) \geq 2^5 \quad\text{for all } s\in(0,1],
\end{equation}
which is incompatible with \eqref{eq: pointwise estimate of U at T}, as desired.
\end{proof}

\section{Global boundedness}
\label{sec: boundedness} 

This section is devoted to global boundedness of radially symmetric solutions to \eqref{sys: ks isp pep} in the case of $\int_\Omega u_0\dd x < 64\pi^2$, 
as a complement of Theorem~\ref{thm: finite-time blowup}.
From an independent interest in chemotactic collapse suggested by~\eqref{eq: chemotactic collapse ?}, we shall establish an $\varepsilon$-regularity derived from Lyapunov functional and Adams inequality.

\subsection{Lyapunov functional and Adams inequality}

The Lyapunov functional of the system \eqref{sys: ks isp ppp} was constructed in~\cite[Proposition~6.1]{Fujie2017}. 
A similar energy dissipation structure exists in \eqref{sys: ks isp pep}.
In the sequel, we may abuse the notation $\mathcal{F}(t) = \mathcal{F}(u,v)$.

\begin{proposition}
  \label{prop: Lyapunov functional}
  Let $(u,v,w)$ be a classical solution given in Proposition~\ref{prop: local existence and uniqueness}. 
  Then the following identity holds:
\begin{equation} 
  \label{eq: energyequation}
\frac{\dd}{\dd t} \mathcal{F}(u(t), v(t))+\mathcal{D}(u(t), v(t))=0 \quad \text {for all } t \in(0, T_{\max}),
\end{equation}
where
\begin{equation}
  \label{sym: mathcalF}
 \mathcal{F}(u, v) 
 :=\int_{\Omega}u \ln u 
 - \int_\Omega  uv
+\frac{1}{2} \int_{\Omega}|\Delta v|^2
+ \frac{\delta}{2}\int_\Omega |\nabla v|^2
\end{equation}
and 
\begin{equation} 
  \label{sym: mathcalD}
  \begin{split}
\mathcal{D}(u, v) 
:= \int_{\Omega} |\nabla v_t|^2
+ \int_{\Omega} u|\nabla(\ln u - v)|^2.
\end{split}
\end{equation}
In particular, 
\begin{equation}
 \mathcal{F}(0) 
 :=\int_{\Omega}u_0 \ln u_0 
 - \int_\Omega  u_0v_0
+\frac{1}{2} \int_{\Omega}|\Delta v_0|^2
+ \frac{\delta}{2}\int_\Omega |\nabla v_0|^2,
\end{equation}
where 
\begin{equation*}
v_0 \in \bigg\{\phi\in W^{2,2}(\Omega): \partial_\nu \phi \equiv 0 \text{ on }\partial\Omega \text{ and } \int_\Omega \phi \dd x = 0\bigg\}
\end{equation*}
is the unique solution to the Neumann problem 
\begin{equation*}
  \begin{cases}
    0 = \Delta v_0 - \frac{2}{\pi^2}\int_\Omega w_0\dd x + w_0, &x\in\Omega,\\
    \partial_\nu v_0 = 0, &x\in\partial\Omega.
  \end{cases}
\end{equation*}
\end{proposition}

The following sharp Adams inequality was established in \cite{Tarsi2012}.

\begin{lemma}
  \label{le: Adams inequality}
  Let $D\subset\mathbb R^4$ be a bounded domain. Then there exists a constant $K>0$ such that 
  \begin{equation*}
    \sup _{u \in H^{2}(D) \cap H_0^{1}(D),\|\Delta u\|_{L^2(D)} \leq 1} 
    \int_{D} e^{32 \pi^2 u^2} \dd x \leq K|D|,
  \end{equation*} 
  and this inequality is sharp, i.e., if $32\pi^2$ is replaced by a larger number, then the supremum is infinite.
\end{lemma}

Inspired by \cite{Nagai1997}, we modify the Adams inequality above.

\begin{lemma}
  \label{le: corollary of Adams inequality}
  Let $D\subset\mathbb{R}^4$ be a bounded domain. Then 
  \begin{equation*}
    \int_D \phi\psi \leq \int_D \phi\ln\phi + \bigg(\int_D \phi\bigg)\ln(K|D| ) + \frac{1}{128\pi^2}\bigg(\int_D\phi\bigg)\|\Delta\psi\|_{L^2(D)}^2  - \int_D\phi\ln\int_D\phi,
  \end{equation*}
  for all nonnegative $\phi\in C^0(\overline{D})$, and $\psi\in H^2(D)\cap H^1_0(D)$. 
  Here, $K > 0$ is given as in Lemma~\ref{le: Adams inequality}. 
\end{lemma}

\begin{proof}
  Without loss of generality, we may assume that $K>1$.
  If $\phi\equiv0$, there is nothing to prove, and therefore, we may assume $\phi\not\equiv0$.
  Using the concavity of the mapping $(0,\infty)\ni \xi\mapsto\ln x\in\mathbb R$, we estimate by Jensen inequality, 
  \begin{equation}
    \label{eq: an application of Jensen inequality}
    \begin{aligned}
      \int_D \phi\psi - \int_D \phi\ln\phi 
      &= \int_D \phi\ln\frac{\e^\psi}{\phi} 
      = \int_D \phi\cdot\int_D\frac{\phi}{\int_D\phi}\ln\frac{\e^\psi}{\phi}\\
      &\leq \int_D\phi \ln\int_D\frac{\e^\psi}{\int_D\phi}
      = \int_D \phi\ln\int_D\e^\psi - \int_D\phi\ln\int_D\phi.
    \end{aligned}
  \end{equation}
  If $\psi\equiv0$, we already finish the proof.
  So we consider only the case $\psi\not\equiv 0$.
  Let $\varphi = \psi/\|\Delta \psi\|_{L^2(D)}$.
  Using Young inequality, we see 
  \begin{equation*}
    \psi = \varphi \cdot \|\Delta \psi\|_{L^2(D)}\leq 32\pi^2\varphi^2 + \frac{1}{128\pi^2}\|\Delta \psi\|_{L^2(D)}^2.
  \end{equation*}
  Then Adams inequality entails 
  \begin{equation}
    \label{eq: an application of Adams inequality}
    \int_D \e^\psi 
    \leq \exp\bigg(\frac{1}{128\pi^2}\|\Delta \psi\|_{L^2(D)}^2\bigg)\cdot\int_D\e^{32\pi^2\varphi^2}
    \leq K|D| \exp\bigg(\frac{1}{128\pi^2}\|\Delta \psi\|_{L^2(D)}^2\bigg).
  \end{equation}
  The desired estimate follows from inserting~\eqref{eq: an application of Adams inequality} into~\eqref{eq: an application of Jensen inequality}.
\end{proof}

\subsection{Boundedness away from the origin}

We collect time-independent pointwise estimates, 
which will be useful.

\begin{lemma}
  \label{le: pointwise estimate}
  Let $(u,v,w)$ be a radially symmetric solution of \eqref{sys: ks isp pep}. 
  Then 
  \begin{equation}
    \label{eq: mass estimates of w}
    \int_\Omega w\dd x \leq \max\bigg\{\int_\Omega w_0\dd x, \frac{1}{\delta}\int_\Omega u_0\dd x\bigg\} 
    := \tilde{m}
    \quad \text{for all } t\in(0,T_{\max}),
  \end{equation}
  and there exists $C_v > 0$ such that 
  \begin{equation}
    \label{eq: pointwise estimates of v}
    |r^3v_r(r,t)| + |r^2v(r,t)| \leq C_v\tilde{m}
    \quad\text{for all } (r,t)\in(0,1)\times(0,T_{\max}).
  \end{equation}
  For any $\rho\in(0,1)$, 
  \begin{equation}
    \label{eq: pointwise estimate of u}
    \|u\|_{L^\infty(\Omega\setminus B_\rho\times(0,T_{\max}))} < \infty.
  \end{equation} 
  In particular, for any $\rho\in(0,1)$,
  \begin{equation}
    \label{eq: pointwise estimates of u and v}
    \|u\|_{C^{2,1}(\overline\Omega\setminus B_\rho\times [T_0,T_{\max}))} 
    + \|v\|_{C^{2,1}(\overline\Omega\setminus B_\rho\times [T_0,T_{\max}))} 
    < \infty
  \end{equation}
  with $T_0 := \min\{T_{\max}/2,1\}$.
\end{lemma}

\begin{proof}
  Integrating $w_t = \Delta w - \delta w + u$ over $\Omega$,
  we have 
  \begin{equation*}
    \frac{\dd}{\dd t}\int_\Omega w\dd x 
    = - \delta \int_\Omega w \dd x + \int_\Omega u_0\dd x
    \quad \text{for all } t\in(0,T_{\max})
  \end{equation*}
  according to the mass conservation property~\eqref{eq: mass conservation}.
  \eqref{eq: mass estimates of w} follows from a comparison argument.
  Multiplying $0 = r^{-3}(v_rr^3)_r - \mu + w$ by $r^3$ and integrating over $(0,r)$ give 
  \begin{equation}
    \label{eq: pointwise estimate of vr}
    |v_rr^3| 
    = \bigg|\int_0^rw\rho^3\dd\rho - \frac\mu4\cdot r^4\bigg|
    \leq \frac{\tilde{m}}{\pi^2}
    \quad\text{for all }(r,t)\in(0,1)\times(0,T_{\max}).
  \end{equation}
  We pick $r_0 = r_0(t)\in[1/2,1]$ with the property 
  \begin{equation*}
    |v(r_0,t)| 
    \leq \frac{1}{|\Omega\setminus B_{1/2}|}\int_{\Omega\setminus B_{1/2}} |v|\dd x 
    \leq \frac{2}{\pi^2}\int_\Omega |v|\dd x 
    \leq C_{\mathrm{e}}\tilde{m}
  \end{equation*}
  for all $t\in(0,T_{\max})$,
  where the constant $C_{\mathrm{e}} > 0$ is provided by the well-known elliptic regularity estimates~\cite{Brezis1973}. Then Newton-Leibniz formula enables us to estimate 
  \begin{equation}
    \label{eq: pointwise estimate of v}
    \begin{split} 
    |v(r,t)| 
    &\leq |v(r,t) - v(r_0,t)| + |v(r_0,t)| 
    \leq \frac{\tilde{m}}{\pi^2}\bigg|\int_{r_0}^r \rho^{-3}\dd\rho\bigg| + C_{\mathrm{e}}\tilde{m}\\
    &\leq \tilde{m}\bigg(r^{-2}+r^{-2}_0+C_{\mathrm{e}}\bigg)
    \leq (5+C_\mathrm{e})\cdot\frac{\tilde{m}}{r^2}
    \end{split}
  \end{equation}
  for all $(r,t)\in(0,1)\times(0,T_{\max})$. 
  So \eqref{eq: pointwise estimates of v} follows from \eqref{eq: pointwise estimate of vr} and \eqref{eq: pointwise estimate of v}. 
  Applying pointwise estimates~\cite[Theorem~1.1]{Winkler2020} in the singular drift-diffusion problem 
  \begin{equation*}
      u_t = \Delta u - \nabla\cdot(u\nabla v)
  \end{equation*}
  with choices of parameters therein $\alpha = 66$, $\beta = 3$, $q=5$ and $K = \tilde{m}^5\pi^{-8}$, 
  we can find $C_u = C_u(\tilde{m}, L) > 0$ such that 
  \begin{equation}
    |x|^{66} u(x,t) \leq C_u
    \quad \text{for all } x\in\Omega 
    \text{ and } t\in(0,T_{\max}),
    \text{ with } L := \sup_{x\in\Omega} (|x|^{66}u_0(x)),
  \end{equation}
  which gives \eqref{eq: pointwise estimate of u}.
  Then we can employ standard parabolic and elliptic regularity theories to assert \eqref{eq: pointwise estimates of u and v}.  
\end{proof}

\subsection{Proof of Proposition~\ref{prop: varepsilon regularity}}

We shall derive Proposition~\ref{prop: varepsilon regularity} by bootstrapping arguments.

\begin{lemma}
  \label{le: entropy finity}
  Let $(u,v,w)$ be a radially symmetric solution of \eqref{sys: ks isp pep}. 
  If there exists $\varrho\in(0,1]$ such that \eqref{eq: supremum on local integral of u} holds, 
  then 
  \begin{equation}
    \label{eq: ulnu is uniform-in-time integrable}
    \sup_{t\in(0,T_{\max})}\int_{\Omega}u\ln (u+1)\dd x 
    + \sup_{t\in(0,T_{\max})}\|w\|_{L^2(\Omega)} < \infty.
  \end{equation}
\end{lemma}

\begin{proof}
  Denote $A$ the supremum in \eqref{eq: supremum on local integral of u}.
  Write 
  \begin{equation*}
    \chi := \frac{1}{2}\cdot\bigg(\frac{64\pi^2}{A} + 1\bigg).
  \end{equation*}
  Then 
  \begin{equation}
    \label{eq: properties of chi}
    \chi > 1
    \quad \text{and} \quad 
    \frac{\chi A}{128\pi^2} < \frac{1}{2}
  \end{equation}
  according to \eqref{eq: supremum on local integral of u}.
  Using Lemma~\ref{le: corollary of Adams inequality} with choices of 
  $D = B_\varrho$, $\phi = u/\chi$ and $\psi = (v-v(\varrho,t))\chi$,
  we get 
  \begin{equation*}
    \begin{split} 
    \int_{B_\varrho} u(v-v(\varrho,t)) 
    &\leq 
    \int_{B_\varrho} \frac{u}{\chi}\ln\frac{u}{\chi} 
    + \bigg(\int_{B_\varrho} \frac{u}{\chi}\bigg)\ln(K|B_\varrho|)\\ 
    &\quad + \frac{1}{128\pi^2}\bigg(\int_{B_\varrho}\frac{u}{\chi}\bigg)\|\chi\Delta (v-v(\varrho,t))\|_{L^2(B_\varrho)}^2  
    - \int_{B_\varrho}\frac{u}{\chi}\ln\int_{B_\varrho}\frac{u}{\chi}\\
    &= \frac{1}{\chi}\int_{B_\varrho}u\ln u 
    + \frac{\chi}{128\pi^2} 
    \cdot \|\Delta v\|_{L^2(B_\rho)}^2 
    \cdot \int_{B_\varrho}u\\
    &\quad + \frac1{\chi}
    \cdot \ln\frac{K\pi^2\varrho^4}{2} 
    \cdot \int_{B_\varrho}u 
    - \frac{\ln\chi}{\chi} \cdot \int_{B_\varrho} u 
    - \int_{B_\varrho}\frac{u}{\chi}
    \cdot \ln\int_{B_\varrho}\frac{u}{\chi}
    \end{split}
  \end{equation*}
  for all $t\in(0,T_{\max})$.
  Lemma~\ref{le: pointwise estimate} entails that 
  \begin{equation*}
    \int_{B_\varrho} uv 
    \leq \frac{1}{\chi}\int_{B_\varrho}u\ln u 
    + \frac{\chi A}{128\pi^2}\cdot \|\Delta v\|_{L^2(B_\varrho)}^2 
    + \varLambda(\varrho)  
  \end{equation*} 
  for all $t\in(0,T_{\max})$,
  with 
  \begin{equation*}
    \varLambda(\varrho) 
    := \frac{A}{\chi}
    \cdot\max\bigg\{\ln\frac{K\pi^2\varrho^4}{2\chi}, 0\bigg\} 
    + \frac{\pi^2\varrho^4}{2\e} + C_vA\tilde{m}\varrho^{-2}.
  \end{equation*}
  Proposition~\ref{prop: Lyapunov functional} gives 
  \begin{equation*}
    \begin{split}
      \int_\Omega u\ln u
      &= \mathcal{F}(t) 
      + \int_\Omega uv 
      - \frac{1}{2}\cdot\|\Delta v\|_{L^2(\Omega)}^2
      - \frac{\delta}{2}\int_\Omega |\nabla v|^2\\
      &\leq \mathcal{F}(0) 
      + \int_{\Omega\setminus B_{\varrho}} uv 
      + \frac{1}{\chi}\int_{B_\varrho}u\ln u 
      + \frac{\chi A}{128\pi^2}\cdot \|\Delta v\|_{L^2(B_\varrho)}^2 
    + \varLambda(\varrho) 
    - \frac{1}{2}\cdot\|\Delta v\|_{L^2(\Omega)}^2\\
    &\leq \mathcal{F}(0) 
    + \frac{1}{\chi}\int_{\Omega}u\ln u 
    + \bigg(\frac{\chi A}{128\pi^2} - \frac{1}{2}\bigg) 
    \cdot \|\Delta v\|_{L^2(\Omega)}^2 
    + C_v\tilde{m}\varrho^{-2}\int_\Omega u_0 
    + \frac{\pi^2}{\chi\e} 
    + \varLambda(\varrho)
    \end{split}
  \end{equation*}
  for all $t\in(0,T_{\max})$. 
  Therefore,
  It follows from \eqref{eq: properties of chi} that 
  \begin{equation}
    \label{eq: ulnu + Delta v2 uniform-in-time integrable}
  \sup_{t\in(0,T_{\max})} \int_\Omega u\ln u 
  + \sup_{t\in(0,T_{\max})} \int_\Omega |\Delta v|^2 
  < \infty.
  \end{equation}
  Noting that 
  \begin{equation*}
    \|w\|_{L^2(\Omega)} 
    \leq \|w-\mu\|_{L^2(\Omega)} 
    + \|\mu\|_{L^2(\Omega)}
    \leq \|\Delta v\|_{L^2(\Omega)} 
    + \pi\|w\|_{L^1(\Omega)}
  \end{equation*}
  for all $t\in(0,T_{\max})$,
  and 
  \begin{equation*}
    \xi\ln(1+\xi) 
    = \xi\ln\bigg(1+\frac{1}{\xi}\bigg) 
    + \xi\ln\xi 
    \leq \xi\ln\xi +1 
    \quad \text{for all } \xi > 0,
  \end{equation*}
  we get \eqref{eq: ulnu is uniform-in-time integrable} from \eqref{eq: ulnu + Delta v2 uniform-in-time integrable} and \eqref{eq: mass estimates of w}.
\end{proof}

\begin{lemma}
  \label{eq: Gagliardo-Nirenberg inequality}
  Let $D\subset\mathbb{R}^4$ be a bounded domain.
  Then there exists $C_{\mathrm{gn}}>0$ such that for any $p\geq1$ and $q>0$, 
  \begin{equation}
    \label{eq: GNI 2+2/p grad2 2/p}
    \|\varphi\|_{L^{2+1/p}(D)}^{2+1/p} 
    \leq C_{\mathrm{gn}}\|\nabla \varphi\|_{L^2(D)}^2\|\varphi\|_{L^{2/p}(D)}^{1/p} 
    + C_{\mathrm{gn}}\|\varphi\|_{L^{2/p}(D)}^{2+1/p} 
  \end{equation}
  and 
  \begin{equation}
    \label{eq: GNI 2+1/p grad2 2/p}
    \begin{split} 
    \|\varphi\|_{L^{2+1/p}(D)}^{2+1/p}
    &\leq C_{\mathrm{gn}}\ln^{-q/p}(1+k)\|\nabla \varphi\|_{L^2(D)}^2\|\varphi\ln^q(\varphi+1)\|_{L^{2/p}(D)}^{1/p} \\
    &\quad + C_{\mathrm{gn}}\|\varphi\|_{L^{2/p}(D)}^{2+1/p}
    + C_{\mathrm{gn}}|D|k^{2+1/p}
    \end{split}
  \end{equation}
  hold for all $k > 0$ and nonnegative function $\varphi\in C^1(\overline\Omega)$.
\end{lemma}

\begin{proof}
\eqref{eq: GNI 2+2/p grad2 2/p} is a direct consequence of  Gagliardo-Nirenberg inequality.
Now we follow the lines in \cite{Biler1994} to show \eqref{eq: GNI 2+1/p grad2 2/p}.
Define 
  \begin{equation*}
    \sigma(\xi) = \sigma_k(\xi) :=
    \begin{cases}
      0, & \xi \in(-\infty, k),\\
      2\xi-2k, & \xi\in[k,2k),\\
      \xi, &\xi\in[2k,\infty).
    \end{cases}
  \end{equation*}
  Then $\sigma(\xi) 
  \leq 
  \min\{\xi, \xi\ln^q(1+\xi)/\ln^q(1+k)\}$
  for $\xi\geq0$ and $q>0$,
  and $\sigma'\in[0,2]$ for a.e. $\xi\in\mathbb{R}$.
  Using \eqref{eq: GNI 2+2/p grad2 2/p}, we can find $C_{\textrm{GN}} > 0$ independent to $p\geq1$, $q>0$ and $k>0$ such that  
  \begin{equation}
    \label{eq: sigma varphi estimates}
    \begin{split}
      \|\sigma(\varphi)\|_{L^{2+1/p}}^{2+1/p} 
      &\leq C_{\textrm{GN}}\|\nabla\sigma(\varphi)\|_{L^2(D)}^2
      \|\sigma(\varphi)\|_{L^{2/p}(D)}^{1/p} 
      + C_{\mathrm{GN}}\|\sigma(\varphi)\|_{L^{2/p}(D)}^{2+1/p} \\ 
      &\leq \frac{4C_{\textrm{GN}}}{\ln^{q/p}(1+k)}
      \|\nabla \varphi\|_{L^2(\Omega)}^2
      \|\varphi\ln^q(1+\varphi)\|_{L^{2/p}}^{1/p} 
      + C_{\textrm{GN}}\|\varphi\|_{L^{2/p}}^{2+1/p}
    \end{split}
  \end{equation}
  for all $\varphi\in C^1(\overline{\Omega})$.
  Since $\xi - \sigma(\xi)\in[0,k]$ for all $\xi\geq0$, 
  we have 
  \begin{equation}
    \label{eq: varphi estimates}
    \begin{split}
      \|\varphi\|_{L^{2+1/p}(D)}^{2+1/p} 
      &\leq 2^{1+1/p}\|\varphi-\sigma(\varphi)\|_{L^{2+1/p}(D)}^{2+1/p} 
      + 2^{1+1/p}\|\sigma(\varphi)\|_{L^{2+1/p}(D)}^{2+1/p} \\
      &\leq 2^{1+1/p}\|\sigma(\varphi)\|_{L^{2+1/p}(D)}^{2+1/p}
      + 2^{1+1/p}|D|k^{2+1/p}.
    \end{split}
  \end{equation}
  Inserting \eqref{eq: sigma varphi estimates} into \eqref{eq: varphi estimates}, we get \eqref{eq: GNI 2+1/p grad2 2/p}.
\end{proof}

\begin{lemma}
  \label{le: Lp finity}
  Let $(u,v,w)$ be a radially symmetric solution of \eqref{sys: ks isp pep}. 
  If \eqref{eq: ulnu is uniform-in-time integrable} holds,
  then 
  \begin{equation}
    \label{eq: up w2p is uniform-in-time integrable}
    \sup_{t\in(0,T_{\max})}\int_{\Omega}u^p\dd x 
    + \sup_{t\in(0,T_{\max})}\int_{\Omega}w^{2p}\dd x 
    < \infty
    \quad\text{for all } p > 1.
  \end{equation}
\end{lemma}

\begin{proof}
  Since 
  \begin{equation*}
    \xi\ln(1+\xi^q) 
    \leq 
    \begin{cases} 
      \ln 2, &\xi\in(0,1),\\
    \xi\ln2 + q\xi\ln\xi, &\xi\in[1,\infty), 
    \end{cases}
  \end{equation*}
  for all $q>0$, 
  \eqref{eq: ulnu is uniform-in-time integrable} implies that 
  \begin{equation*}
    \Upsilon = \Upsilon(p) := \sup_{t\in(0,T_{\max})}\int_{\Omega}u\ln (u^{p/2}+1)\dd x 
    + \sup_{t\in(0,T_{\max})}\|w\|_{L^2(\Omega)} < \infty
  \end{equation*} 
  holds for all $p>1$.
  Without loss of generality, we may assume that $p > 2$.
  Integrating by parts, we compute 
  \begin{equation}
    \label{eq: d/dt int up}
    \begin{split}
      \frac{1}{p-1}\frac{\dd}{\dd t}\int_\Omega u^p 
      &= \frac{p}{p-1}\int_\Omega u^{p-1}\nabla\cdot(\nabla u - u\nabla v)\\
      &= - p\int_\Omega u^{p-2}|\nabla u|^2 
      + p\int_\Omega u^{p-1}\nabla u\cdot\nabla v\\
      &= -\frac{4}{p}\int_\Omega|\nabla u^{p/2}|^2 
      + \int_\Omega \nabla u^p\cdot\nabla v 
      = -\frac{4}{p}\int_\Omega|\nabla u^{p/2}|^2 
      - \int_\Omega u^p\Delta v \\
      &\leq -\frac{4}{p}\int_\Omega|\nabla u^{p/2}|^2 
      + \int_\Omega u^p w \\
      &\leq -\frac{4}{p}\int_\Omega|\nabla u^{p/2}|^2 
      + \eta^{-2p}\int_\Omega u^{p + 1/2} 
      + \eta\int_\Omega w^{2p+1}
    \end{split}
  \end{equation}
  for all $t\in(0,T_{\max})$ and $\eta \in (0,1)$, 
  where we have used the second equation in \eqref{sys: ks isp pep}.
  Integrating by parts, we estimate 
  \begin{equation}
    \label{eq: d/dt int w2p}
    \begin{split}
      \frac{1}{2p-1}\frac{\dd}{\dd t}\int_\Omega w^{2p} 
      &= \frac{2p}{2p-1}\int_\Omega w^{2p-1}(\Delta w - \delta w + u)\\
      &= - 2p\int_\Omega w^{2p-2}|\nabla w|^2 
      - \frac{2p\delta}{2p-1}\int_\Omega w^{2p}
      + \frac{2p}{2p-1}\int_\Omega w^{2p-1}u\\
      &\leq - \frac{2}{p}\int_\Omega |\nabla w^p|^2 
      - \delta\int_\Omega w^{2p} 
      + \eta\int_\Omega w^{2p+1} 
      + 2^{p}\eta^{1/2-p}\int_\Omega u^{p+1/2}
    \end{split}
  \end{equation}
  for all $t\in(0,T_{\max})$ and $\eta \in (0,1)$.
  Write 
  \begin{equation*}
    y_p(t) := \frac{1}{p-1}\int_\Omega u^p 
    + \frac{1}{2p-1}\int_\Omega w^{2p}
    \quad\text{for } t\in(0,T_{\max}).
  \end{equation*}
  Then collecting \eqref{eq: d/dt int up} and \eqref{eq: d/dt int w2p} yields 
  \begin{equation}
    \label{eq: ODI of y_p}
    \begin{split}
      y_p' + \delta y_p 
      &\leq -\frac{4}{p}\int_\Omega|\nabla u^{p/2}|^2 
        + 2\eta^{-2p}\int_\Omega u^{p + 1/2} 
        + \delta\int_\Omega u^p\\ 
      &\quad - \frac{2}{p}\int_\Omega |\nabla w^p|^2 
      + 2\eta\int_\Omega w^{2p+1}
    \end{split}
  \end{equation}
  for all $t\in(0,T_{\max})$ and $\eta\in(0,1/2)$.
  Applying \eqref{eq: GNI 2+2/p grad2 2/p} with $\varphi = w^p$, we have 
  \begin{equation}
\begin{split}
      \int_\Omega w^{2p+1} 
      &\leq C_{\mathrm{gn}}\|\nabla w^p\|_{L^2(\Omega)}^2\|w\|_{L^2(\Omega)} 
      + C_{\mathrm{gn}}\|w\|_{L^2(\Omega)}^{p+1/2}\\
      &\leq C_{\mathrm{gn}}\Upsilon \|\nabla w^p\|_{L^2(\Omega)}^2 
      + C_{\mathrm{gn}}\Upsilon^{p+1/2}
\end{split}
  \end{equation}
  for all $t\in(0,T_{\max})$.
Using~\eqref{eq: GNI 2+1/p grad2 2/p} with $\varphi = u^{p/2}$ and $q=p/2$, we get  
\begin{equation}
  \label{eq: estimate up+1/2 by GNI}
  \begin{split} 
  \int_\Omega u^{p+1/2}
  &\leq C_{\mathrm{gn}}\ln^{-1/2}(1+k)\|\nabla u^{p/2}\|_{L^2(\Omega)}^2
  \|u\ln(u^{p/2}+1)\|_{L^1(\Omega)}^{1/2} \\
  &\quad + C_{\mathrm{gn}}\|u\|_{L^{1}(\Omega)}^{p+1/2}
  + C_{\mathrm{gn}}\pi^2k^{2+1/p}\\
  &\leq C_{\mathrm{gn}}\Upsilon^{1/2}\ln^{-1/2}(1+k)\|\nabla u^{p/2}\|_{L^2(\Omega)}^2
  + C_{\mathrm{gn}}\|u_0\|_{L^{1}(\Omega)}^{p+1/2}
  + C_{\mathrm{gn}}\pi^2k^{2+1/p}
  \end{split}
\end{equation}
for all $t\in(0,T_{\max})$ and $k>0$.
Fix 
\begin{equation*}
  \eta := \frac{1/p}{C_{\mathrm{gn}}\Upsilon + 2}
\end{equation*}
and 
\begin{equation*}
  k := \exp\bigg(\frac{p^2C_{\mathrm{gn}}^2\Upsilon}{\eta^{4p}}\bigg).
\end{equation*}
Collecting \eqref{eq: ODI of y_p}--\eqref{eq: estimate up+1/2 by GNI}, we have 
\begin{equation}
  \label{eq: refined ODI of y_p}
  \begin{split}
    y_p' + \delta y_p 
    &\leq - 2\eta^{-2p}\int_\Omega u^{p + 1/2} 
    + \delta\int_\Omega u^p\\
    &\quad + 2\eta C_{\mathrm{gn}}\Upsilon^{p+1/2}
    + 2\eta^{-2p}C_{\mathrm{gn}}m^{p+1/2}
  + 2\eta^{-2p}C_{\mathrm{gn}}\pi^2k^{2+1/p}\\
  &\leq \pi\delta^{2p+1} 
  + 2\eta C_{\mathrm{gn}}\Upsilon^{p+1/2}
    + 2\eta^{-2p}C_{\mathrm{gn}}m^{p+1/2}
  + 2\eta^{-2p}C_{\mathrm{gn}}\pi^2k^{2+1/p}
  \end{split}
\end{equation}
for all $t\in(0,T_{\max})$.
Hence, \eqref{eq: up w2p is uniform-in-time integrable} follows from Gronwall inequality applied to \eqref{eq: refined ODI of y_p}.
\end{proof}

Now we are in a position to show Proposition~\ref{prop: varepsilon regularity}.

\begin{proof}[Proof of Proposition~\ref{prop: varepsilon regularity}]
  In view of Lemma~\ref{le: entropy finity} and Lemma~\ref{le: Lp finity},
  we have 
  \begin{equation*}
    \sup_{t\in(0,T_{\max})}\int_\Omega w^4\dd x < \infty,
  \end{equation*}
  which implies 
  \begin{equation*}
    \sup_{t\in(0,T_{\max})}\int_\Omega |\nabla v|^5 \dd x 
    < \infty,
  \end{equation*}
  by standard elliptic regularity theories. 
  A direct application of well-established Neumann heat semigroup estimates~\cite[Lemma~1.3]{Winkler2010b} shows \eqref{eq: L infinity u}.
\end{proof}

\section{Chemotactic collapse} 
\label{sec: chemotactic collapse}

Now we are in a position to show Proposition~\ref{prop: chemotactic collapse}.

\begin{proof}[Proof of Proposition~\ref{prop: chemotactic collapse}]
  As 
  \begin{equation*}
    \begin{split}
      \bigg|\frac{\dd}{\dd t}\int_{B_r}u(\cdot,t)\dd x \bigg|
      &= 2\pi^2\Big|\int_0^r(u_r\rho^3)_r - (uv_r\rho^3)_r\dd\rho\Big|\\
      &\leq 2\pi^2|u_r(r,t)|r^3 + |u(r,t)v_r(r,t)|r^3
    \end{split}
  \end{equation*}
  for all $(r,t)\in(0,1)\times(T_{\max}/2,T_{\max})$, 
  we have 
  \begin{equation*}
    \sup_{t\in(T_{\max}/2, T_{\max})}\bigg|\frac{\dd}{\dd t}\int_{B_r}u(\cdot,t)\dd x \bigg| < \infty
    \quad \text{for all } r\in(0,1),
  \end{equation*}
  according to Lemma~\ref{le: pointwise estimate}.
  Then the function
  \begin{equation*}
    \Xi(r) 
    := \lim_{t\uparrow T_{\max}}\int_{B_r} u(\cdot,t)\dd x 
    \quad \text{for } r\in(0,1)
  \end{equation*}
  is well-defined.
  As $\Xi(r)$ is non-decreasing, we put 
  \begin{equation}
    \label{sym: m_star}
    m_\star := \lim_{r\downarrow0}\Xi(r).
  \end{equation}
  We claim that $m_\star \geq 64\pi^2$. 
  Indeed, if $m_\star < 64\pi^2$, then there exists $\varrho\in(0,1)$ such that 
  \begin{equation*}
    \Xi(\varrho) < 64\pi^2.
  \end{equation*}
  Moreover, there exists $T\in(0,T_{\max})$ such that 
  \begin{equation*}
    \sup_{t\in(T,T_{\max})} \int_{B_\varrho} u(\cdot,t)\dd x 
    < 64\pi^2.
  \end{equation*}
  Recalling Proposition~\ref{prop: varepsilon regularity}, 
  this warrants 
  \begin{equation*}
    \sup_{t\in(T,T_{\max})}\|u\|_{L^\infty(\Omega)} < \infty,
  \end{equation*}
  which is incompatible with \eqref{eq: extensibility principle}, as $T_{\max} < \infty$. 

  As it follows from Lemma~\ref{le: pointwise estimate} that 
  \begin{equation*}
    \|u\|_{C^1(\Omega\setminus B_\rho)\times(T_{\max}/2,T_{\max})} < \infty 
    \quad \text{for all } \rho\in(0,1),
  \end{equation*}
  we may define the nonnegative function 
  \begin{equation*}
    f(x) := \lim_{t\uparrow T_{\max}} u(x,t) 
    \quad \text{for all } x\in\overline{\Omega}\setminus\{0\}.
  \end{equation*}
  Then $u(\cdot, t)\to f(\cdot)$ in $C^0_{\loc}(\overline{\Omega}\setminus\{0\})$ as $t\uparrow T_{\max}$.
  Fatou lemma and mass conservation formula~\eqref{eq: mass conservation} entail that $f\in L^1(\Omega)$.

  For $\phi\in C^0(\overline{\Omega})$, we calculate 
  \begin{equation*}
    \begin{split}
      &\quad \Big|\int_\Omega u(x,t)\phi(x)\dd x 
      - m_\star \phi(0) 
      - \int_\Omega f(x)\phi(x)\dd x\Big| \\
      &= \Big|\int_{B_r}u(x,t)(\phi(x)-\phi(0))\dd x 
      + \Big(\int_{B_r}u(x,t)\dd x -m_\star \Big) \phi(0)\\
      &\quad + \int_{\Omega\setminus B_r}(u(x,t) - f(x))\phi(x)\dd x
      - \int_{B_r}f(x)\phi(x)\dd x\Big|\\
      &\leq \|u_0\|_{L^1(\Omega)}\|\phi(x)-\phi(0)\|_{C^0(B_r)} 
      + \Big|\int_{B_r}u(x,t)\dd x -m_\star \Big|
      \cdot \|\phi\|_{L^\infty(\Omega)} \\
      &\quad + \|u(x,t) - f(x)\|_{L^\infty(\Omega\setminus B_r)}\|\phi\|_{L^1(\Omega)}
      + \|f\|_{L^1(B_r)}\|\phi\|_{L^\infty(\Omega)}
    \end{split}
  \end{equation*}
  for all $r\in(0,1)$.
  Then we have 
  \begin{equation*}
    \begin{split}
      &\quad \limsup_{t\uparrow T_{\max}}\Big|\int_\Omega u(x,t)\phi(x)\dd x 
      - m_\star \phi(0) 
      - \int_\Omega f(x)\phi(x)\dd x\Big| \\
      &\leq \|u_0\|_{L^1(\Omega)}\|\phi(x)-\phi(0)\|_{C^0(B_r)} 
      + (\Xi(r) -m_\star) \|\phi\|_{L^\infty(\Omega)} 
      + \|f\|_{L^1(B_r)}\|\phi\|_{L^\infty(\Omega)}
    \end{split}
  \end{equation*}
  for all $r\in(0,1)$. 
  Taking the limit $r\downarrow 0$, we get 
  \begin{equation*}
    \begin{split}
      \limsup_{t\uparrow T_{\max}}\Big|\int_\Omega u(x,t)\phi(x)\dd x 
      - m_\star \phi(0) 
      - \int_\Omega f(x)\phi(x)\dd x\Big| 
      \leq 0,
    \end{split}
  \end{equation*}
  thanks to $\phi\in C^0(\overline{\Omega})$, $f\in L^1(\Omega)$ and the definition~\eqref{sym: m_star} of $m_\star$.
  This is exactly \eqref{eq: weak star convergence of u}.
\end{proof}

\section*{Acknowledgments}
The first author has been supported by ``the Fundamental Research Funds for the Central Universities'' (No.~B250201215).
The last author has been supported in part by National Natural Science Foundation of China (No. 12271092, No. 11671079).

\section*{Declarations}
\begin{itemize}
  \item Conflict of interest: The authors declare they have no financial interests.
  \item Data availability: Not applicable.
  \end{itemize}

\providecommand{\bysame}{\leavevmode\hbox to3em{\hrulefill}\thinspace}
\providecommand{\MR}{\relax\ifhmode\unskip\space\fi MR }
\providecommand{\MRhref}[2]{%
  \href{http://www.ams.org/mathscinet-getitem?mr=#1}{#2}
}
\providecommand{\href}[2]{#2}

\end{document}